\documentclass[11pt]{article}

\usepackage{amsmath}
\usepackage{amsfonts}
\usepackage{theorem}
\usepackage{graphics}
\usepackage{amssymb}
\usepackage{tikz}
\usetikzlibrary{intersections,decorations.pathreplacing,positioning,quotes,angles}
\usepackage{arcs}
\usepackage{color}

\newtheorem{Theorem}{Theorem}[section]
\newtheorem{Lemma}{Lemma}[section]
\newtheorem{Proposition}[Lemma]{Proposition}
\newtheorem{Corollary}[Lemma]{Corollary}

\newtheorem{Remark}{Remark}[section]

\newenvironment{proof}{\paragraph*{Proof}}{\hfill$\square$}

\newcommand{\ovec}{\overrightarrow}

\medskip

\date{}

\begin{document}

\begin{center}
{\large \bf Almost Lyapunov Functions for Nonlinear Systems}
\end{center}

\centerline{\scshape Shenyu Liu
}
\medskip
{\footnotesize
 \centerline{Coordinated Science Laboratory}
\centerline{University of Illinois at Urbana-Champaign}
\centerline{Urbana, IL 61801, USA}
}

\medskip

\centerline{\scshape Daniel Liberzon}
\medskip
{\footnotesize
 \centerline{Coordinated Science Laboratory}
\centerline{University of Illinois at Urbana-Champaign}
\centerline{Urbana, IL 61801, USA}
}

\medskip

\centerline{\scshape Vadim Zharnitsky }

\medskip
{\footnotesize
\centerline{ (corresponding author, email: {\tt vzh@uiuc.edu})}
 \centerline{Department of Mathematics }
 \centerline{University of Illinois at Urbana-Champaign}
 \centerline{Urbana, IL 61801, USA}
 \centerline{and}
\centerline{ZJU-UIUC Institute, International Campus}
\centerline{Zhejiang University}
\centerline{ Haining, China}
}

\vspace{20mm}


\begin{abstract}

 We study convergence  of nonlinear systems in the presence of an ``almost Lyapunov'' function which, unlike the classical Lyapunov function, is allowed to be nondecreasing---and even increasing---on a nontrivial subset  of the phase space.  Under the assumption that the vector field is free of  singular points (away from the origin) and that the subset where the Lyapunov function does not decrease is sufficiently small, we prove that solutions approach a small neighborhood of the origin. A nontrivial  example where this theorem applies is constructed.

\end{abstract}

\section{Introduction}\label{introduction}

For general nonlinear systems, asymptotic stability is typically shown through Lyapunov's direct method (see, e.g., \cite{KH02}), which involves constructing a Lyapunov function $V$ whose time derivative along solutions is negative except at the equilibrium. Even if this property holds for the nominal system, stability is not guaranteed when there is a perturbation because $V$ might not necessarily decrease along solutions of the perturbed system. One natural way to address this issue is to find another Lyapunov function $W$ for this perturbed system by perturbing $V$ accordingly; this is known as the Zubov method \cite{RD65} on which there are many recent results such as \cite{FC01},\cite{SD02}. On the other hand, if
it is desirable to  use the same candidate Lyapunov function $V$, one may hope to establish stability, at least in some weaker sense, if the measure of the set  where $V$ is not decreasing along perturbed solutions is relatively small. We call such a candidate Lyapunov function ``{\em almost Lyapunov}" in this paper. \\

Besides the above applications for stability of perturbed systems, almost Lyapunov functions can be useful when computational complexity is the main difficulty. While it is straightforward  to compute the derivative of an arbitrary Lyapunov function along solutions, it might be quite challenging to analytically check the sign of this derivative either for all states, or just for a region of interest. For example,  in the case when both the differential equation  and the Lyapunov function are polynomials of high degree, the derivative is also a polynomial and verifying stability reduces to checking whether a polynomial is negative definite. This problem is computationally hard, as it is related to Hilbert's 17th problem \cite{BR00} and is an important subject of current research (see, e.g., \cite{GC11},\cite{GB12}). Following existing techniques, we may be able to verify that the time derivative of $V$ is negative only in a proper subset of the region of interest, while  not in the entire region. This demonstrates the  need for tools that would let one conclude stability if $V$ is only an ``almost Lyapunov" function, which is studied in this paper. \\

When a general candidate Lyapunov function is constructed, the sign of its derivative along solutions can also be checked by techniques based on random sampling \cite{RT12} instead of deterministic methods. This approach only requires one to verify that the derivative is negative at a sequence of states picked randomly inside the region. One can use the Chernoff bound (see, e.g., \cite{RT12},\cite{MV02}) to characterize the number of such sample points needed to obtain a reliable upper bound on the relative measure of points in the region of interest for which the desired inequality can possibly fail. Hence the problem is again converted into finding an ``almost Lyapunov" function.
There is not much work related to this topic of ``almost Lyapunov" functions. Before our preliminary work \cite{DL14} and \cite{SL16}, the most relevant work is \cite{AB69} and its extension \cite{AA11}, both of which use higher order derivatives of Lyapunov functions for stability analysis. Nevertheless, although a relatively small measure of the set of states where the Lyapunov function does not decrease is implied in both papers, none of them explicitly uses this fact.  \\

When working with ``almost Lyapunov" functions, we encounter  regions in the state space where the system trajectories might temporarily diverge (in the sense of growth of Lyapunov function). Nevertheless, our main result shows that when the volume of the "bad" region where $V$ does not decrease fast  is sufficiently small, the system is stable in the following weaker sense as characterized by three properties: 1. Every solution starting within a region that is slightly smaller than the region of interest will remain in the region of interest; 2. All such solutions will converge to a small region containing the equilibrium, with a uniform  bound in time; 3. Once they reach this small region around the equilibrium, solutions  will remain  there afterwards. The differences between the sizes of the respective regions depend on the measure of the  bad set, and they compensate for possible temporary overshoots.

The first result of this type was obtained in~\cite{DL14} by using a perturbation argument. In that paper, an arbitrary solution was compared with a solution that avoided "bad regions" and converged to the equilibrium. Then, using continuous dependence of solutions on initial conditions, it was found that this arbitrary solution will not end up too far from the equilibrium.

In this paper we present a different approach, which is based on the geometry of curves in the Euclidean space. The basic idea here (following up on our preliminary work~\cite{SL16}) is that in order to accumulate a net gain in $V$ along a solution, the tubular neighborhood swept out by a ball of a certain radius moving along this solution trajectory needs to be contained inside the region where $V$ does not decrease fast enough. Consequently, if such ``bad'' regions are not big enough, $V$ cannot increase overall (even though a temporary gain is still possible). To illustrate this type of system behavior, we construct an example in which  there is a small region where the time derivative of $V$ is positive and to which our main result applies.

The paper is mainly organized in the following order: Frequently used terms and variables are defined in Section \ref{preliminaries}. Our main result (Theorem \ref{thm1}) is stated in Section \ref{mainresult}. Its proof is given in Section \ref{proofoftheorem}. Section \ref{sec:GUAS} presents a global result on system stability which can be derived from almost Lyapunov function and Section \ref{nontrivialexample} contains a numerical example where our theorem is applied on with some discussion. After Section \ref{conclusion} concludes the paper, the previous result from \cite{DL14} is briefly mentioned in Appendix A and the proof of an auxiliary result (Proposition \ref{non-overlap}) is provided in Appendix B.

\section{Preliminaries}\label{preliminaries}
Consider a general system
\begin{equation}\label{def:sys}
\dot{x}=f(x),\quad x\in\mathbb{R}^n,\quad f(0) =0,
\end{equation}
where $f:\mathbb R^n\to\mathbb R^n$ is a Lipschitz function. Consider a function $V:\mathbb{R}^n\to[0,\infty)$ which is positive definite and $C^1$ with locally Lipschitz gradient, which we denote by $V_x$. We say it is a \textit{Lyapunov function} for the system \eqref{def:sys} if \begin{equation}\label{ldm}
\dot{V}(x):=V_x(x)\cdot f(x)<0\quad \forall x\neq 0
\end{equation}
The system \eqref{def:sys} can be shown to be asymptotically stable if such a Lyapunov function exists \cite[Ch.~4]{KH02}. A stronger version of Lyapunov function is when $V$ decays at a certain positive rate $a$:
\[
\dot{V}(x) < -aV(x)\quad\forall x\neq 0
\]
While this property needs not to hold on the entirely region of interest $D$, we set
\begin{equation}\label{lcvdot}
\Omega:=\{x\in D:\dot V(x)\geq -aV(x)\}
\end{equation}
and when the measure of $\Omega$ is ``small", we informally say this $V$ is an almost Lyapunov function for the system \eqref{def:sys} because now
\[
\dot{V}(x)< -aV(x)\quad \forall x\in D\backslash\Omega.
\]
Notice that the solution trajectory passing through $\Omega$ does not necessarily imply  growth of $V$; it is only in the subset   $\{ x\in \Omega: \dot {V}(x)> 0 \}$ that growth of $V$  occurs. In this paper, we take the region $D$ to be of the following form:
\begin{equation}\label{def:D}
D:=\{x\in\mathbb R^n:c_1\leq V(x)\leq c_2\}, \quad c_2>c_1>0
\end{equation}
We assume $D$ to be compact\footnote{This is true when $V$ is radially unbounded. Otherwise the results of our theorem are still applicable if the initial state of the system is inside a compact connected component of $D$. In this case we take this compact connected component as the region of interest $D$.
}. We refer to $f$ as ``non-vanishing'' when
\begin{equation}\label{nonvanishing}
f(x)\neq 0\quad\forall x\in D.
\end{equation}
The non-vanishing condition  clearly requires the equilibrium at origin to be excluded from $D$. Next define
\begin{equation}\label{def:b}
b:=\max_{x\in D} \dot V(x).
\end{equation}
Finally, let $B_{\gamma}^n(x)$ be the closed ball whose center is at $x$ in $\mathbb{R}^n$ with radius $\gamma$. Also define the function
$\text{vol}(\cdot)$ to be the standard volume function induced by the Euclidean metric. Recall that a general expression for the volume of a $n$-dimensional ball of radius $\gamma$ is:
\begin{equation}\label{eqn:volume}
\text{vol}(B_{\gamma}^{n})=\frac{\pi^{\frac{n}{2}}}{\Gamma(\frac{n}{2}+1)}\gamma^{n}=:\chi(n)\gamma^n
\end{equation}
where $\Gamma$ is the standard gamma function \cite[Ch.~4.11]{CR89}. More  notations will be introduced in the course of the proof.

\section{Main result}\label{mainresult}

We are now ready to state our main result:

\begin{Theorem}\label{thm1}
Consider a system (\ref{def:sys}) with a locally Lipschitz right-hand side $f$, and a function $V:\mathbb R^n\to[0,+\infty)$ which is positive definite and $C^1$ with locally Lipschitz gradient. Let the region $D$ be defined via (\ref{def:D}) with some fixed  $c_1 < c_2$ and assume it is compact.

Let  $\dot{V}(x)< -aV(x)\ \forall\, x\in D\backslash\Omega$   for some $a>0$, where  $\Omega\subset D$ is a measurable set, let $\dot V(x) \geq -aV(x)\ \forall\, x \in \Omega$, and let $f$ be non-vanishing in $D$ as defined in (\ref{nonvanishing}).
Assume
\[
{\max_{x\in D} \dot V(x)} < a \min_{x\in D} V,
\]
i.e. $b < a c_1$ where $b$ is defined in \eqref{def:b}.

Then there exist constants $\bar\epsilon>0$, $g>0$, $h>0$ such that  for any $\epsilon\in[0,\bar\epsilon)$,  if  $\textnormal{vol}(\Omega^*)\leq\epsilon$ for every connected component $\Omega^*$ of $\Omega$ , then there exists $T\geq 0$ so that for any i.c. $x_0\in D$ with $V(x_0)< c_2-h\epsilon^{\frac{1}{n}}-g\epsilon$, the solution $x(t)$ of (\ref{def:sys})
stays in the domain $D$ for all time, i.e. $x(t)\in D$ for all $t\geq 0$,  and
\[
V(x(t))\leq c_1+h\epsilon^{\frac{1}{n}}+g\epsilon
\]
 for  all $ t\geq T$.

\end{Theorem}

\begin{Remark}
The results of Theorem \ref{thm1} are illustrated in Figure \ref{illustration_fig}. As seen from the figure, the proof will actually give slightly sharper estimates than what is stated in the theorem, namely, $x(t) \leq V(x_0)+g\epsilon$ for all $t\geq 0$ and $V(x(T))\leq c_1+h\epsilon^{\frac{1}{n}}$.
The term  $ h\epsilon^{\frac{1}{n}}$ serves as a  ``buffer"  ensuring that  the  solution tube is always in $D$ while the term  $g\epsilon$ is a threshold for possible transient overshoot. The exact formulas for $g,h$ will be given by \eqref{g},\eqref{h} respectively and $\bar\epsilon$ will be explicitly found in Section~\ref{sec:4.3}. Later in the proof of the main theorem the reader will also see that the convergence before time $T$ is in fact exponential, in the form of
\begin{equation*}
V(x(t)) \leq (V(x(0))+ \frac{g}{2}\epsilon) e^{-\lambda(\epsilon)t^*}+\frac{g}{2}\epsilon,
\end{equation*}
where $\lambda(\epsilon)$ is a positive, continuous and strictly decreasing function on $[0,\bar\epsilon)$ with $\lambda(0)<a$ and some $t^*\in[\max\{0,t-\frac{2g\epsilon}{b}\},t]$.
\end{Remark}
\begin{Remark}
In the limit $\epsilon\rightarrow 0$, the almost Lyapunov function becomes the standard Lyapunov function and the theorem gives the usual conclusion that one could expect from the  Lyapunov stability theory. In particular, any solution starting at the higher  level set $V=c_2$ will converge to the lower level set $V=c_1$.
\end{Remark}

\begin{figure}
\centering
\includegraphics[scale=0.5]{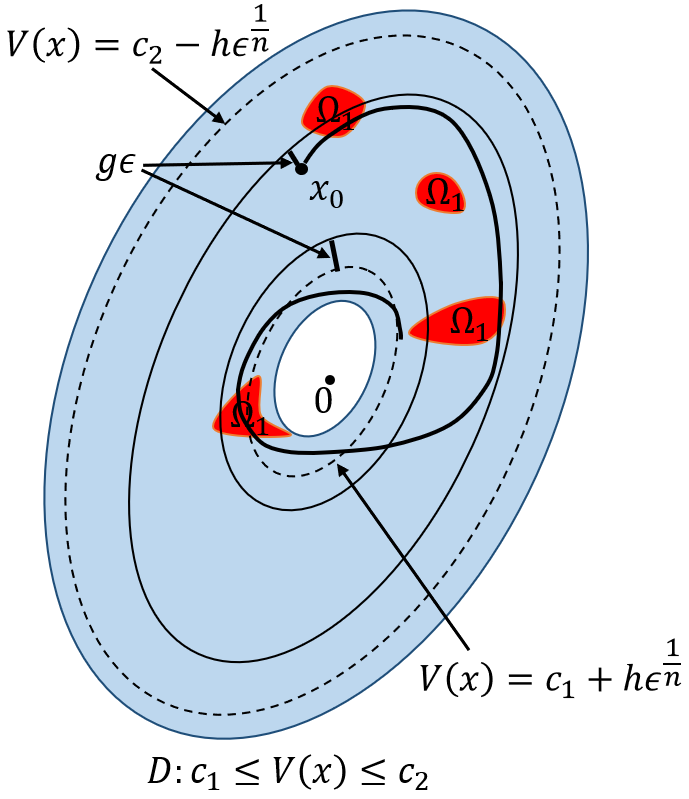}
\caption{Illustration of Theorem \ref{thm1}}\label{illustration_fig}
\end{figure}

\section{Proof of theorem}\label{proofoftheorem}

 The main idea of the proof  relies on the following observation: if the measure of  $\Omega$ is small enough, there will be too little time for a tube around the   solution  to stay inside $\Omega$ so the growth of $V$ could not be accumulated. The proof contains 4 major steps: 
\begin{enumerate}
\item The first step is to show that when the time derivative of $V$ is positive, the solution has to be in  a connected component  $\Omega^*$ and a tube around the solution is contained in $\Omega^*$.
\item The second step is to use a non-self-overlapping condition to compute an upper bound on the time that the solution stays in $\Omega^*$ based on the volume swept out by the solution tube.
\item  The next step is to find a bound on the change of $V$ over the time estimated in the previous step. We will   conclude that when the volume of $\Omega^*$ is sufficiently  small, the change of $V$ will be negative.
\item The last step generalizes previously obtained estimates to the possible scenario of repeated passage of the solution  through several, or even  infinitely many,  connected components of $\Omega$. By connecting segments of the solution, we argue that although there might be temporary overshoots in $V$, overall the solution will converge to a smaller sub-level set.
\end{enumerate}

\subsection{Estimates on the solution tube}
Since $f$ is a Lipschitz function and $D$ is compact, we can define the following bounds:
\begin{equation}\label{bar_L0}
\bar{L}_0:=\max_{x\in D}|f(x)|,
\end{equation}
\begin{equation}\label{underline_L0}
\underline{L}_0:=\min_{x\in D}|f(x)|.
\end{equation}
Note that the vector field $f$ is non-vanishing in $D$ if and only if $\underline{L}_0>0$. Let $L_1$ be the Lipschitz constant of $f$ over $D$:
\begin{equation}\label{L_1}
|f(x_1)-f(x_2)|\leq L_1|x_1-x_2|\quad\forall x_1,x_2\in D.
\end{equation}
In addition, since $V$ is assumed to be $C^1$ and has  locally Lipschitz gradient, we also define some bounds on $V_x$:
\begin{equation}\label{M_1}
M_1:=\max_{x\in D}|V_x(x)|,
\end{equation}
and $M_2$ be the Lipschitz constant of $V_x$ over $D$:
\begin{equation}\label{M_2}
|V_x(x_1)-V_x(x_2)|\leq M_2|x_1-x_2|\quad\forall x_1,x_2\in D.
\end{equation}
For $\eta\in[0,1]$, we pick a connected component from the following set
\begin{equation}\label{Omega_eta}
\{x\in D:\dot{V}(x)\geq-\eta a V(x)\}
\end{equation}
where $a$ comes from the hypothesis of the theorem. We call such a connected set $\Omega_\eta$. By this definition, $\Omega_1$ is the same as $\Omega^*$, a connected component of $\Omega$. By choosing an appropriate family of  connected components, it is possible to achieve that if
$\eta_1\leq \eta_2$ then $\Omega_{\eta_1}\subseteq\Omega_{\eta_2}$. \\
The next three lemmas establish existence of  a disk of positive  radius that is sweeping through $\Omega_1$ along the solution forming a tube that is contained  inside $\Omega_\eta$:
\begin{Lemma}\label{lem:0}
For any $x_1,x_2\in D$,
\[
|V(x_1)-V(x_2)|\leq M_1|x_1-x_2|.
\]
\end{Lemma}
\begin{proof}
If the line segment between $x_1,x_2$ entirely lies in $D$, by Mean Value Theorem there exists $x_3$ on the segment such that $V(x_2)=V(x_1)+V_x(x_3)\cdot(x_2-x_1)$. Now by \eqref{M_1},
\[
|V(x_1)-V(x_2)|=|V_x(x_3)\cdot(x_1-x_2)|\leq |V_x(x_3)||x_1-x_2|\leq M_1|x_1-x_2|.
\]
In the case when the line segment is partially outside of $D$, let us say say that $y_1,y_2\in \partial D$ are two points on the segment connecting $x_1,x_2$ such that the line segment between $y_1,y_2$ is outside $D$. Since $y_1,y_2$ are on the boundary of $D$, the $V$ value must be either $c_1$ or $c_2$ at these two points. If $V(y_1)\neq V(y_2)$, say $V(y_1)=c_1$ and $V(y_2)=c_2$, then $V(x)\leq c_1$ or $V(x)\geq c_2$ for all $x$ on the line segment from $y_1$ to $y_2$. This cannot happen since $V$ is a continuous function. Therefore $V(y_1)=V(y_2)$. Hence using triangle inequality,
\begin{align*}
|V(x_1)-V(x_2)|&=|(V(x_1)-V(y_1))+(V(y_2)-V(x_2))|\\
&\leq |V(x_1)-V(y_1)|+|V(y_2)-V(x_2)|\\
&\leq M_1|x_1-y_1|+M_1|y_2-x_2|\\
&\leq M_1|x_1-x_2|.
\end{align*}
The second to last inequality follows  from the fact that the two segments $x_1$ to $y_1$ and $x_2$ to $y_2$ are contained  in $D$ so we can apply our earlier result. The last inequality is simply the fact that the sum of the lengths of the two segments is no longer than the total distance between $x_1$ and $x_2$. In the case when there are multiple segments between $x_1$ and $x_2$ that are outside of $D$, repeating the above analysis on each interval, we still get the same result.
\end{proof}

\begin{Lemma}\label{lem:1}
For any $x_1,x_2\in D$,
\begin{equation}\label{eqn:lem1}
|\dot{V}(x_1)-\dot{V}(x_2)|\leq \alpha|x_1-x_2|,
\end{equation}
where $\alpha:=M_1L_1+M_2\bar{L}_0$.
\end{Lemma}
\begin{proof}
Estimate
\begin{align*}
|\dot{V}(x_1)-\dot{V}(x_2)|&=|V_x(x_1)f(x_1)-V_x(x_2)f(x_2)|\\
&\leq|V_x(x_1)||f(x_1)-f(x_2)|+|f(x_2)||V_x(x_1)-V_x(x_2)|\\
&\leq M_1|f(x_1)-f(x_2)|+\bar{L}_0|V_x(x_1)-V_x(x_2)|\\
&\leq M_1L_1|x_1-x_2|+\bar{L}_0M_2|x_1-x_2|\\
&=\alpha|x_1-x_2|.
\end{align*}
Notice that we have used the definitions of $M_1$ from  \eqref{M_1} and $\bar L_0$ from \eqref{bar_L0} in the second to last inequality and the two Lipschitz constants $L_1,M_2$ from \eqref{L_1},\eqref{M_2} in the last inequality.

\end{proof}

\begin{Lemma}\label{lem:2}
If  $x\in\Omega_\eta$ then $\left(B_{\gamma_\eta}^n(x)\cap D
\right)\subseteq\Omega_1$, where
\begin{equation}\label{gamma}
\gamma_\eta:=\frac{(1-\eta)ac_1}{\alpha+\eta a M_1}
\end{equation}
with  $\alpha$ as defined in Lemma \ref{lem:1}.
\end{Lemma}
\begin{proof}
Let  $x\in\Omega_\eta,y\in D$ be such that $|x-y|\leq\gamma_\eta$. Since both of them are in $D$, by Lemma \ref{lem:0}, $V(x)\leq V(y)+M_1|x-y|\leq V(y)+M_1\gamma_\eta$.
Therefore
\begin{align*}
\dot V(y)&\geq\dot V(x)-|\dot V(x)-\dot V(y)|\geq-\eta a V(x)-\alpha|x-y|\\
&\geq-\eta a (V(y)+M_1\gamma_\eta)-\alpha\gamma_\eta=-\eta a V(y)-(1-\eta)a c_1\\
&\geq -a V(y).
\end{align*}
In the second inequality we have used the fact that $x\in\Omega_\eta$ so $\dot V(x)\geq \eta a V(x)$. We also used the result from Lemma \ref{lem:1} for bounding the second term in this step. Lemma \ref{lem:0} is used in the third inequality. Across the second line the terms depending on $\gamma_\eta$ are collected together and substituted with its definition \eqref{gamma}. In the last inequality we have used the fact that $y\in D$ so $c_1\leq V(y)$. Hence we have shown $y\in\Omega_1$ and $\left(B_{\gamma_\eta}^n(x)\cap D\right)\subseteq\Omega_1$.

\end{proof}

Define the normal disk of radius $\gamma$ centered at $x$ to be
\begin{equation}\label{normal_plane}
N_\gamma(x)=\{y\in B_{\gamma}^n(x):(y-x)\cdot f(x)=0\},
\end{equation}
which  is a ball $B_{\gamma}^{n-1}(x)$  in the hyperplane
\[
  \{y\in {\mathbb R^n}:(y-x)\cdot f(x)=0\}.
\]

Define
\begin{equation}\label{SweptArea}
S_{\eta,(s,t)}= \underset{ \tau\in(s,t)  }{\cup}  N_{\gamma_\eta}(x(\tau))
\end{equation}
to be the tube of radius $\gamma_\eta$ around  the solution on the time interval  $s$ to $t$. We will often refer to it as the {\it solution tube}.
We will say the tube is \textit{non-self-overlapping} over time interval $(s,t)$ if
\begin{equation}\label{sof}
N_{\gamma_\eta}(x(\tau_1))\cap N_{\gamma_\eta}(x(\tau_2))=\emptyset\quad\forall \tau_1,\tau_2\in(s,t),\tau_1\neq\tau_2.
\end{equation}

In a non-self-overlapping tube all the states are swept out only once by such $N_{\gamma_\eta}(x(\tau))$ normal disk at some $\tau\in(s,t)$. There will be
more discussion of non-self-overlapping condition in the next subsection.

Let
\[
\mathcal{L}_{s}^{t}:=\int_s^t|f(x(\tau)|d\tau
\]
be the length of the solution trajectory from time $s$ to $t$. Using the  bounds \eqref{bar_L0} and \eqref{underline_L0} on $f$, one has
\begin{equation}\label{length_boundary}
\underline{L}_0(t-s)\leq\mathcal{L}_{s}^{t}\leq \bar{L}_0(t-s).
\end{equation}
Define
\begin{equation}\label{g}
g:=\frac{b}{\underline L_0 \text{vol}(B^{n-1}_{\gamma_\eta})},
\end{equation}
\begin{equation}\label{h}
h:=M_1\chi(n)^{-\frac{1}{n}},
\end{equation}
where $\chi(n)$ comes from \eqref{eqn:volume}. Define a shrunk domain
\[
D^*:=\{x\in \mathbb R^n:c_1+h\epsilon^{\frac{1}{n}}\leq V(x)\leq c_2-h\epsilon^{\frac{1}{n}}\}.
\]
For any initial state $x(0)=x_0\in D$ with $V(x_0)< c_2-g\epsilon-h\epsilon^{\frac{1}{n}}$, by the standard theory of ODEs the solution can be continued either indefinitely or to the boundary of $D^*$. Define
\begin{equation}\label{def:T}
T:=\inf\{\tau\geq 0:x(t)\not\in D^*\}
\end{equation}
By this definition, $T=0$ if $V(x_0)< c_1+h\epsilon^{\frac{1}{n}}$ and $T$ could also be infinite if the solution stays in $D^*$ forever. Eventually, in the proof we show that $T$ has to be finite and it is impossible for the solution to reach the outer boundary of $D^*$ with $V(x(T))=c_2-h\epsilon^{\frac{1}{n}}$.
This $T$ will be the one in the main theorem statement that we are looking for. \\

Define the subset of the time interval  when the solution stays in  $\Omega_{\eta}$ as
\begin{equation}\label{def:X_eta}
X_{\eta} = \{ \tau \in [0,T): x(\tau)\in \Omega_{\eta}   \}.
\end{equation}
While the set $X_{\eta}$ might have a complicated structure, the relevant  part for us  is the interior which must be a union of intervals.  The almost Lyapunov function might increase
 when the solution is considered over such an interval.  When the solution is considered over a subset of $X_{\eta}$ which has empty interior, the almost Lyapunov function
 will be decreasing with the  rate $a$. A {\em  maximal interval} contained in $X_{\eta}$  is an interval in $X_{\eta}$ which cannot be enlarged without leaving $X_{\eta}$. We
 will also refer to such intervals  as {\em  connected components} of $X_{\eta}$.

\begin{figure}
\centering
\includegraphics[scale=0.5]{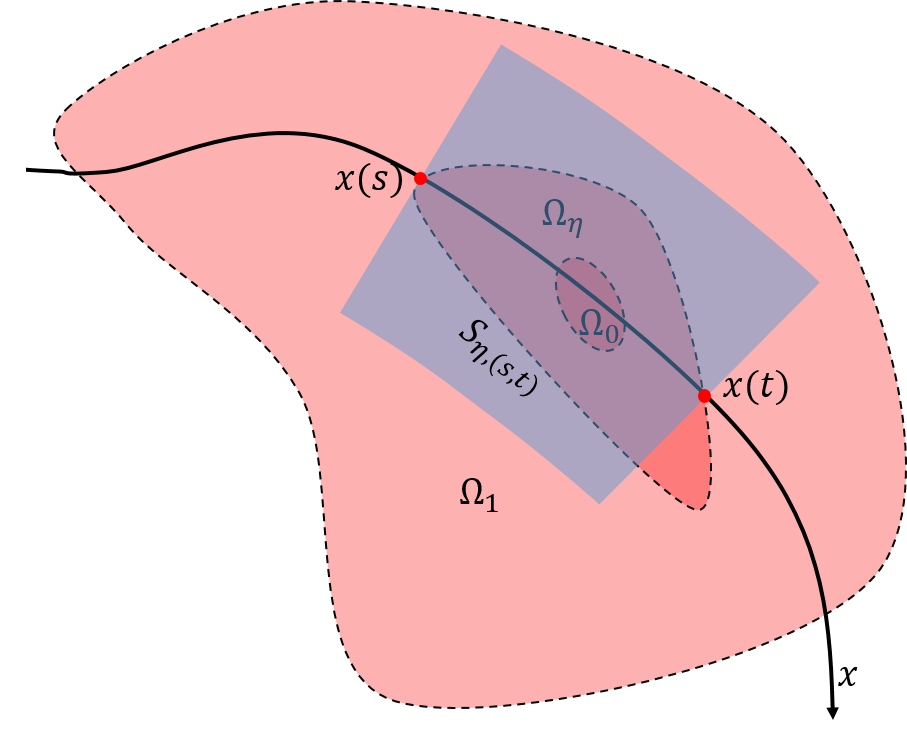}

\caption{A planar example showing the solution trajectory passing through $\Omega_\eta$, generating a tubular neighborhood $S_{\eta,(s,t)}$. In higher dimension the set  $S_{\eta,(s,t)}$ would look like a cylinder.}\label{fig:traj}
\end{figure}

The sweeping tube $S_{\eta,(s,t)}$ generated over a connected component $(s,t)\subseteq X_\eta$ is illustrated in Figure \ref{fig:traj}. Intuitively the volume of $S_{\eta,(s,t)}$ is the cross-section area times the trajectory length over $(s,t)$. The next lemma proves this, under the assumption that there is no self-overlapping:

\begin{Lemma}\label{prop:vol}
If the solution is non-self-overlapping over time interval $(s,t)$, then
\begin{equation}\label{rho}
\normalfont\text{vol}(S_{\eta,(s,t)})=\chi(n-1)\gamma_{\eta}^{n-1}\mathcal{L}_{s}^{t}.
\end{equation}
\end{Lemma}

The proof of this lemma is a direct application of results  from \cite[Chapter 4.10]{CR89},\cite{RF06}. The conditions for non-self-overlapping will be discussed in the next section.

\begin{Remark}
The formula in \cite{RF06} yields a signed volume with multiplicity (which is a result of negative self-overlapping); nevertheless, the non-self-overlapping condition we have ensures that there are no negative or multiple counts of the integrated volume and the result is indeed the absolute volume that we want as a lower bound.
\end{Remark}


\begin{Lemma}\label{cor:5}
$S_{\eta,(s,t)}\subseteq \Omega_1$ for all $(s,t)\subset X_{\eta}$.
\end{Lemma}
\begin{proof}
By Lemma \ref{lem:2}, the definition of $S_{\eta,(s,t)}$ in \eqref{SweptArea} and the definition of $X_\eta$ in \eqref{def:X_eta}, it suffices to show that $B_{\gamma_\eta}(x)\subseteq D$ for any $x\in D^*\cap \Omega_{\eta}$. If this is not true, there exists $x\in D^*\cap \Omega_{\eta}$ such that $B_{\gamma_\eta}(x)$ is partially outside of $D$ (cannot be completely outside of $D$ as $x\in D^*\subset D$). In this case, introduce the sets
\begin{align*}
S_{in}&=\partial B_{\gamma_\eta}(x)\cap D,\\
S_{out}&=\partial B_{\gamma_\eta}(x)\backslash S_{in},\\
S_{D}&=\partial D\cap B_{\gamma_\eta}(x).
\end{align*}
None of these sets are  empty and  for any $y\in S_{D}$, $V(y)=c_1$ or $c_2$. By definition of $D^*$ and Lemma \ref{lem:0} we have
\[
h\epsilon^{\frac{1}{n}}\leq |V(x)-V(y)|\leq M_1|x-y|\Rightarrow |x-y|\geq \left(\frac{\epsilon}{\chi(n)}\right)^{\frac{1}{n}}.
\]
Let $z\in S_{out}$. Then the line segment $[x,z]$ intersects with $S_{D}$ at some point $y$ so $|x-z|=|z-y|+|y-x|$ and then
\[
 \gamma_\eta\geq \delta+\left(\frac{\epsilon}{\chi(n)}\right)^{\frac{1}{n}}
\] for some $\delta>0$. Denote the volume bounded by the surfaces $S_{in},S_{D}$ by $A$. Then $A\subseteq\Omega_1$ so $\text{vol}(A)\leq \text{vol}(\Omega_1)\leq \epsilon$. On the other hand, by the earlier analysis points on $S_{D}$ are at least $\left(\frac{\epsilon}{\chi(n)}\right)^{\frac{1}{n}}$ away from $x$ and points on $S_{in}$ are at least $\delta+\left(\frac{\epsilon}{\chi(n)}\right)^{\frac{1}{n}}$ away from $x$. This means $A$ contains a ball of radius $\left(\frac{\epsilon}{\chi(n)}\right)^{\frac{1}{n}}$ so $\text{vol}(A)>\chi(n)\left(\frac{\epsilon}{\chi(n)}\right)=\epsilon$ (the positivity of  $\delta$ and the continuity of the surface result in the strict inequality), which is a contradiction.

\end{proof}

The result of Lemma~\ref{cor:5} is illustrated in Figure \ref{fig:traj} that the sweeping tube is a subset of the ``bad region" $\Omega_1$. Now applying the formula \eqref{rho} here with the assumption that the solution is non-self-overlapping, we have
\begin{eqnarray}\label{newrho}
\epsilon\geq\normalfont\text{vol}(\Omega_1)\geq\text{vol}(S_{\eta,(s,t)})=\text{vol}(B_{\gamma_\eta}^{n-1})\mathcal{L}_{s}^{t} \\
\geq\text{vol}(B_{\gamma_\eta}^{n-1}) \, \underline{L}_0(t-s)=\frac{b}{g}(t-s). \nonumber
\end{eqnarray}

\begin{Corollary}\label{lem:last}
Let $(s,t)\subset X_{\eta}$ and assume the solution over this time interval is non-self-overlapping. Then the length of the time interval must satisfy
\[
t-s\leq \frac{g\epsilon }{b}.
\]
\end{Corollary}

\subsection{On non-self-overlapping condition}

The following proposition  gives  a geometric criterion of non-self-overlapping.
\begin{Proposition}\label{non-overlap}
Consider a tube of radius $\rho_0$ around a space curve $\gamma(\tau)$ whose radius of curvature is bounded from below by $\rho$. If $ \rho > \rho_0$ and if the length $\mathcal L$ of $\gamma(\tau)$ is bounded:
\begin{equation}\label{eqn:prop1}
\mathcal L<2 \rho\left(\pi-\sin^{-1}(\frac{\rho_0}{\rho})\right)
\end{equation}
then the tube is non-self-overlapping.
\end{Proposition}

The value on the right hand side of \eqref{eqn:prop1} is the curve length of a circular arc with radius of curvature $\rho$ and chord distance of $2\rho_0$ between end points. The proof of this proposition makes use of two classical results of Fenchel's Theorem and Schur's Comparison Theorem (see \cite{JS08}), and is provided in the appendix B.

At this point, the solution of our system can be viewed as a space curve $x=\gamma(s)$ in $\mathbb R^n$. Thus we have the curvature
\[
k(s) = \frac{[\gamma',\gamma'']}{|\gamma'|^3}(s),
\]
where $[*,*]$ is a standard area form.
This formula is a simple consequence of  the definition of centripetal acceleration $a = v^2 k$. Indeed, $[\gamma', \gamma''] =
|\gamma'| |\gamma''| \sin \alpha$ where $\sin\alpha$ is the angle between the two vectors $\gamma',\gamma''$. When $[\gamma', \gamma'']$ is divided by $|\gamma'|^3$, we obtain $|\gamma''|\sin \alpha/|\gamma'|^2$, which is the projection
of acceleration onto the normal vector to the curve (centripetal acceleration) divided by velocity squared. The second order derivative in the definition of $\kappa(s)$ involves gradient of $f(x)$, which may not exist if $f(x)$ is only assumed to be Lipschitz. Nevertheless, according to Rademacher's Theorem, a Lipschitz vector field is differentiable almost everywhere so curvature exists almost everywhere, which is enough for our subsequent proof as discussed in \cite{JS08} and the result is similar to the case if the curve is $C^2$. Hence, applying this bound to our curve $x(s)$ wherever $\nabla f$ exists:
\[
|k(s)|  \leq \frac{|\dot x| |\ddot x|}{|\dot x|^3} \leq \frac{|\ddot x|}{|\dot x|^2} \leq \frac{||\nabla f(x)|| \,  |\dot x|}{|\dot x |^2} \leq \frac{||\nabla f(x)||}{|f(x) | }   \leq \frac{L_1}{\underline{L}_0}.
\]
This implies that $\frac{L_1}{\underline{L}_0}$ is an upper bound of curvature along the solution $x(t)$ almost everywhere. Therefore, since radius of curvature is simply the reciprocal of curvature, Proposition \ref{non-overlap} implies a sufficient condition for non-self-overlapping solution of our system:
\begin{Corollary}\label{prop:sof}
A tube of radius $\gamma_{\eta}$ around the solution $x(\tau)$  is non-self-overlapping over the interval $(s,t)$ if
\begin{equation}\label{turning_radius}
\gamma_\eta<\frac{\underline{L}_0}{L_1}
\end{equation}
and
\begin{equation}\label{overall_length}
\mathcal{L}_{s}^{t}< \frac{2\underline{L}_0}{L_1}\left(\pi-\sin^{-1}(\frac{L_1\gamma_\eta}{\underline{L}_0})\right).
\end{equation}
\end{Corollary}

Note that according to  \eqref{gamma} $\gamma_\eta$ is a decreasing function of $\eta$ and $\gamma_1=0$, thus, the inequality \eqref{turning_radius} can always be satisfied by picking $\eta$ close enough to $1$.

\begin{Remark}
Bounded curvature is an important feature for non-vanishing vector fields since bounded curvature  prevents the system from some undesired  behavior which will not generate new sweeping volume, such as spinning around inside a small region.
\end{Remark}

Now we have found a criterion of non-self-overlapping  \eqref{overall_length} in terms of the constraint on the path length,   but we need to reformulate this criterion in terms of the measure of the bad set. Suppose that \eqref{turning_radius} holds with the volume bound analogue of  \eqref{overall_length}
\begin{equation}\label{overall_volume}
\epsilon<\epsilon_1:= \text{vol}(B_{\gamma_\eta}^{n-1})\frac{2\underline{L}_0}{L_1}\left(\pi-\sin^{-1}(\frac{L_1\gamma_\eta}{\underline{L}_0})\right).
\end{equation}
Then we have
\begin{Lemma}\label{cor_3}
Assume $\eta$ satisfies the inequality \eqref{turning_radius} and $\epsilon<\epsilon_1$ as defined in \eqref{overall_volume}. Then $S_{\eta,(s,t)}$ is non-self-overlapping for any $(s,t)\subseteq X_{\eta}$.
\end{Lemma}

\begin{proof}

By Lemma \ref{cor:5} we have $S_{\eta,(s,t)}\subseteq \Omega_1$ so that $\text{vol}(S_{\eta,(s,t)})\leq \text{vol}(\Omega_1)\leq \epsilon<\epsilon_1$. Let
\[
\tilde t:=\sup\{\tau\in(s,t]:\mbox{solution is non-self-overlapping over }[s,\tau)\}.
\]
The solution is always non-self-overlapping when $\tau$ is sufficiently close to $s$ because of the inequality \eqref{turning_radius} so the above set is non-empty and the supremum exists. Our goal is to show $\tilde{t}=t$. Because \eqref{overall_length} means any tube generated by any shorter curve will be non-self-overlapping, the solution is non-self-overlapping over $[s,\tau)$ for all $\tau\in(s,\tilde t)$. Thus by the continuity of $\text{vol}(S_{\eta,(s,\tau)})$ with respect to $\tau$,
\begin{align*}
\text{vol}(B_{\gamma_\eta}^{n-1})\mathcal{L}_{s}^{\tilde t}&=\lim_{\tau\to\tilde t^-}\left(\text{vol}(B_{\gamma_\eta}^{n-1})\mathcal{L}_{s}^{\tau}\right)=\lim_{\tau\to \tilde t^-}\text{vol}(S_{\eta,(s,\tau)})\\
&=\text{vol}(S_{\eta,(s,\tilde t)})\leq \text{vol}(S_{\eta,(s,t)})<\epsilon_1\\
&= \text{vol}(B_{\gamma_\eta}^{n-1})\frac{2\underline{L}_0}{L_1}\left(\pi-\sin^{-1}(\frac{L_1\gamma_\eta}{\underline{L}_0})\right)
\end{align*}
\[
\Rightarrow \mathcal{L}_{\check s}^{\tilde t}<\frac{2\underline{L}_0}{L_1}\left(\pi-\sin^{-1}(\frac{L_1\gamma_\eta}{\underline{L}_0})\right).
\]
If $\tilde t\neq t$, then since $\mathcal L_{s}^\tau$ is a continuous and strictly increasing function of $\tau$ (because of non-vanishing vector field), we can always pick $t^*\in (\tilde t,t)$ such that
\[
\mathcal{L}_s^{\tilde t}<\mathcal{L}_{s}^{t^*}<\frac{2\underline{L}_0}{L_1}\left(\pi-\sin^{-1}(\frac{L_1\gamma_\eta}{\underline{L}_0})\right).
\]
Hence by Corollary \ref{prop:sof} we conclude that the solution is non-self-overlapping up to time $t^*$, which contradicts maximality of $\tilde t$. Thus $\tilde t=t$.
\end{proof}

\subsection{Change of $V$ when passing through $\Omega_\eta$}\label{sec:4.3}

We now specify the threshold $\bar\epsilon$ in the statement of Theorem 1
\[
\bar\epsilon:=\min\{\epsilon_1,\epsilon_2\},
\]
where $\epsilon_1$ is defined in \eqref{overall_volume} and
\begin{equation}\label{e2}
\epsilon_2:=\text{vol}(B_{\gamma_\eta}^{n-1}) \, \frac{\underline{L}_0(b+\eta ac_1)^2}{\alpha \bar{L}_0b}.
\end{equation}
Note that when $\eta<1$, we have $\gamma_\eta>0$ and thus both $\epsilon_1,\epsilon_2$ are positive, which implies $\bar \epsilon >0$. In addition, when \eqref{turning_radius} is satisfied and $\epsilon<\bar \epsilon$, $S_{\eta,(s,t)}$ is non-self-overlapping for any $(s,t)\in X_\eta$ by Lemma \ref{cor_3}. Hence by Corollary \ref{lem:last} we have
\begin{equation}\label{con:teta2}
t-s\leq \frac{g\epsilon}{b}<\frac{g\bar \epsilon}{b}\leq \frac{g\epsilon_2}{b}=\frac{(b+\eta ac_1)^2}{\alpha \bar{L}_0b}.
\end{equation}
These inequalities in \eqref{con:teta2} are essential and will be repeatedly used in the proofs of subsequent lemmas. \\
We now show  that  $V$ will always decrease over any connected component of $X_\eta$ excluding those containing
boundary points  $\tau=0$ and $\tau= T$, if the latter exists. When the solution passes through the connected component containing the initial point $\tau=0$ or  $\tau= T$ then $V$ may actually increase but is bounded by a fixed value. This is summarized in the next lemma:

\begin{Lemma}\label{lem:5}
 Assume $\eta\in(0,1)$ satisfies \eqref{turning_radius} and $\epsilon<\bar \epsilon$. For any connected component $(s,t)\subset X_\eta$, define $\Delta V_{(s,t)}:=V(x(t))-V(x(s))$. Then
\begin{enumerate}

\item If $s=0$ and $V(x(0))< c_2-h\epsilon^{\frac{1}{n}}-g\epsilon$,
\[
\Delta V_{(s,t)}\leq \left\{\begin{array}{cl}
g\epsilon & \mbox{ if }t=T,\\
\frac{g}{2}\epsilon & \mbox{ if } t\neq T.
\end{array}\right.
\]
\item If $s>0$ and $V(x(s))< c_2-h\epsilon^{\frac{1}{n}}-\frac{g}{2}\epsilon$,
\[
\Delta V_{(s,t)}\leq \left\{\begin{array}{cl}
\frac{g}{2}\epsilon & \mbox{ if }t=T,\\
\phi(t-s) & \mbox{ if } t\neq T.
\end{array}\right.
\]
where
\begin{equation}\label{phi}
\phi(\tau):=\left\{\begin{array}{ccc}\frac{1}{4}\tau^2\alpha \bar{L}_0 -\tau \eta ac_1&\mbox{\rm if}&\tau \alpha \bar{L}_0<2(b+\eta ac_1),\\b\tau-\frac{(b+\eta ac_1)^2}{\alpha \bar{L}_0}&\mbox{\rm if}&\tau \alpha \bar{L}_0\geq 2(b+\eta ac_1).
\end{array}\right.
\end{equation}

\end{enumerate}
\end{Lemma}

\begin{Remark}
We  observe that when  $(t-s) \alpha \bar{L}_0<2(b+\eta ac_1)$, $b$ does not appear in the bound for $\Delta V_{(s,t)}.$ This corresponds to the case when the bound $b$ is too loose, or the upper bound of $\dot{V}$ is unknown or not pre-determined. We have done studies of such less constrained almost Lyapunov functions previously and an example on which the theorem is applicable is not found yet.
\end{Remark}

\begin{proof}
The proof consists of four steps. \\

\noindent
{\bf Case 1: ($s=0$ and $t=T$)}. \\ Notice $\Delta V_{(s,t)}=\int_s^t\dot V(x(\tau))d\tau\leq \int_s^tbd\tau=b(t-s)\leq g\epsilon$ for any $(s,t)\subset X_\eta$. The last inequality comes from Corollary \ref{lem:last}. Thus $g\epsilon$ is an upper bound for $\Delta V_{(s,t)}$  for any connected components $(s,t)$ in $X_\eta$, in particular for the special case when both $s=0$ and $t=T$. \\

\noindent
{\bf Case 2: ($s=0$ and $t\neq T$)}. \\ In this case $t$ is finite. Since $(s,t)$ is a maximal interval, either $x(t)\in\partial\Omega_\eta$ or $x(t)\in\partial D^*$, the boundary of $D^*$. If it is the latter one, we are only interested in the case when $\Delta V_{(0,t)}>0$, that is, the case $V(x(t))=c_2-h$. Notice that in this case $\Delta V_{(0,t)}=V(x(t))-V(x(0))> (c_2-h\epsilon^{\frac{1}{n}})-(c_2-h\epsilon^{\frac{1}{n}}-g\epsilon)=g\epsilon$. This contradicts with the general upper bound of $g\epsilon$ on $\Delta V_{(s,t)}$ derived in Case 1. Thus we must have $x(t)\in\partial\Omega_\eta$ so $\dot V(x(t))=-\eta a V(x(t))\leq -\eta a c_1$. Next we compute a tighter upper bound on $\Delta V_{(0,t)}$. It follows from (\ref{eqn:lem1}) that for any $t_1,t_2\in [s,t]$,
\begin{align}\label{bound:Vdot}
|\dot{V}(x(t_1))-\dot{V}(x(t_2))|&\leq\alpha|x(t_1)-x(t_2)|\nonumber\\
&\leq\alpha\int_{t_1}^{t_2}|f(x(\tau))|d\tau \leq\alpha \bar{L}_0|t_1-t_2|.
\end{align}
Thus,    $\dot{V}$, when considered as  a function of time, is a Lipschitz function with  Lipschitz constant  $\alpha \bar{L}_0$. We can now estimate $\Delta V_{(0,t)}$ by collecting inequalities:
\begin{multline}\label{max1}
\Delta V_{(0,t)}=\int_{0}^{t}\dot V(x(\tau))d\tau\\
\mbox{with the bounds }  t <\frac{(b+\eta a c_1)^2}{\alpha \bar L_0 b}, \dot V(x(t))\leq-\eta a c_1, \dot V(x(t_0))\leq b,\\
|\dot V(x(t_1))-\dot V(x(t_2))|\leq \alpha \bar L_0|t_1-t_2|\ \forall t_0,t_1,t_2\in [0,t].
\end{multline}
The first bound comes from \eqref{con:teta2}  and the other bounds have been introduced earlier.  We claim that a necessary condition for the inequalities in \eqref{max1} to hold is:
\[
\dot V(x(\tau))\leq \min\{b,\alpha\bar L_0 (t-\tau)-\eta a c_1\},
\]
where the first bound $b$ in the min function above is immediate. The second bound in the min function comes from  $\dot V(x(t))\leq -\eta a c_1$ and the Lipschitz bound on $\dot V$. Hence we conclude that its integration gives an upper bound for $\Delta V_{(0,t)}$:
\begin{align*}
\Delta V_{(0,t)} &\leq  \int_0^{t}\min\{b,\alpha \bar L_0(t-\tau)-\eta a c_1\}d\tau\\
&=\int_0^{t}\min\{b,\alpha \bar L_0\tau-\eta a c_1\}d\tau
\end{align*}

A change of variable is used for deriving the second line above. Notice that the minimum function switches value when $b=\alpha \bar L_0\tau-\eta a c_1$, that is, when $\tau=\frac{b+\eta a c_1}{\alpha \bar L_0}$.
To estimate  the integral,  consider first the case when $t \geq \frac{b+\eta a c_1}{\alpha \bar L_0}$. In this case
\begin{align*}
\Delta V_{(0,t)}&\leq \int_0^{\frac{b+\eta a c_1}{\alpha \bar L_0}}(\alpha \bar L_0\tau-\eta a c_1)d\tau+\int_{\frac{b+\eta a c_1}{\alpha \bar L_0}}^{t} bd\tau\\
&=\frac{1}{2}\alpha \bar L_0\left(\frac{b+\eta a c_1}{\alpha \bar L_0}\right)^2-\eta a c_1\frac{b+\eta a c_1}{\alpha \bar L_0}+b\left( t-\frac{b+\eta a c_1}{\alpha \bar L_0}\right)\\
&=b t+\frac{(b+\eta a c_1)^2-2\eta a c_1(b+\eta a c_1)-2b(b+\eta a c_1)}{2\alpha \bar L_0}\\
&=b t-\frac{(b+\eta a c_1)^2}{2\alpha \bar L_0}
< b t-\frac{b t}{2} \leq \frac{g}{2}\epsilon.
\end{align*}
The two inequalities on the last line come from the inequalities in \eqref{con:teta2}. Now, if  $t  < \frac{b+\eta a c_1}{\alpha \bar L_0}$, there is no switch and we only need to evaluate one integral:
\begin{align*}
\Delta V_{(s,t)} &\leq   \int_0^t (\alpha \bar L_0 s-\eta a c_1)ds= \frac 1 2 \alpha \bar L_0 t^2 -\eta a c_1  t\\
&= \left(\frac 1 2 \alpha \bar L_0 t -\eta a c_1 \right) t<\left(\frac 1 2 \alpha \bar L_0 \left(\frac{b+\eta a c_1}{\alpha \bar L_0}\right) -\eta a c_1 \right)t\\
&=\frac{1}{2}(b-\eta ac_1)t<\frac{b}{2}t\leq \frac{g}{2}\epsilon.
\end{align*}
The last inequality above comes from \eqref{con:teta2}. Thus we have shown that $\frac{g}{2}\epsilon$ is an upper bound for $\Delta V_{(s,t)}$ when $s=0,t\neq T$. \\

\noindent
{\bf Case 3: ($s\neq 0,t=T$)}  \\ We start by considering any connected component $(s,t)$ such that $s\neq 0$. Again because it is maximal, we can only have $x(s)\in\partial \Omega_\eta$. This is because $x(s)\in\partial D^*$ is impossible as otherwise $x(\tau)\not\in D^*$ for some $\tau<s$. Thus we should have $\dot V(x(s))=-\eta a V(x(s))\leq -\eta a c_1$. Similar to \eqref{max1}, we obtain a system of inequalities
\begin{multline}\label{max2}
\Delta V_{(s,t)}=\int_{s}^{t}\dot V(x(\tau))d\tau\\
\mbox{with bounds  } t -s <\frac{(b+\eta a c_1)^2}{\alpha \bar L_0 b},\dot V(x(s))\leq-\eta a c_1, \dot V(x(t_0))\leq b,\\
|\dot V(x(t_1))-\dot V(x(t_2))|\leq \alpha \bar L_0|t_1-t_2|\ \forall t_0,t_1,t_2 \in[s,t],
\end{multline}
where the first bound again comes from \eqref{con:teta2}. The bounds are essentially  the same as \eqref{max1} but with the only difference that the  boundary condition is  $\dot V(x(s))\leq-\eta a c_1$ instead of $\dot V(x(t))\leq-\eta a c_1$. By symmetry considerations (change of variables $\tau'=t+s-\tau$ and then shift the time so $s=0$), the upper bound will be the same and, thus,  we have $\Delta V_{(s,t)}< \frac{g}{2}\epsilon$. This proves the special case when $\tau=T$, if $T< \infty$. \\

\noindent
{\bf Case 4: ($s\neq 0,t\neq T$)} \\ From the analysis in case 3 we see that $V(x(t))=V(x(s))+\Delta V_{(s,t)}< (c_2-h-\frac{g}{2}\epsilon)+\frac{g}{2}\epsilon=c_2-h$. Hence $x(t)\not\in\partial D^*$. So by maximality of $(s,t)$ we must have both $x(s),x(t)\in \partial \Omega_\eta$. Therefore, we have the system of inequalities
\begin{multline}\label{max3}
\Delta V_{(s,t)}=\int_{s}^{t}\dot V(x(\tau))d\tau\\
\mbox{with bounds } (t-s) < \frac{(b+\eta a c_1)^2}{\alpha \bar L_0 b}, \dot V(x(s)),\dot V(x(t))\leq-\eta a c_1,\\
\dot V(x(\tau))\leq b, |\dot V(x(t_1))-\dot V(x(t_2))|\leq \alpha \bar L_0 |t_1-t_2|\ \forall \tau,t_1,t_2 \in[s,t].
\end{multline}
By the same reasoning as we did  for \eqref{max1}, we have the following bound as a necessary condition:
\begin{equation}\label{eqn_trapezoid}
\dot V(\tau)\leq \min\{b,\alpha \bar L_0(\tau- s)-\eta a c_1,\alpha \bar L_0(t-\tau)-\eta a c_1\}
\end{equation}
for all $\tau\in[s,t]$. Hence
\begin{align*}
\Delta V_{(s,t)}&\leq\int_s^t \min\{b,\alpha \bar L_0(\tau - s)-\eta a c_1,\alpha \bar L_0(t - \tau)-\eta a c_1\}d\tau\\
&=\int_0^{t-s}\min\{b,\alpha \bar L_0 \tau -\eta a c_1,\alpha \bar L_0(t-s-\tau)-\eta a c_1\}d\tau.
\end{align*}

\begin{figure}
\centering
\begin{tikzpicture}[scale=0.9]
\tikzset{mypoints/.style={red}}
\def\ptsize{2.0pt}
\draw[thick,->][name path=xaxis] (-0.5,0) -- (6.5,0) node[anchor=north west] {$\tau$};
\draw[thick,->][name path=yaxis] (0,-4) -- (0,2.5) node[anchor=south east] {$\dot{V}$};
\draw[dashed][name path=b](0,2) -- (6.5,2);
\draw[dashed](0,-3) -- (6,-3);
\draw[name path=traj](0,-3) -- (2,2) -- (4,2) -- (6,-3);
\draw[dashed](6,-3) -- (6,0);
\path[red,name intersections={of=traj and xaxis}];
\coordinate [label=below right:$s$](1) at (0,0);
\coordinate [label=below right:$t$](4) at (6,0);
\coordinate [label=left:$b$](b) at (0,2);
\coordinate [label=left:$-\eta a c_1$](s) at (0,-3);
\foreach \p in {1,4}
\fill[mypoints] (\p) circle (\ptsize);
\end{tikzpicture}
\centering
\caption{ Upper bound of $\dot{V}$ vs. $\tau$ on the trajectory passing through $\Omega_\eta$ }\label{fig:vdott}
\end{figure}
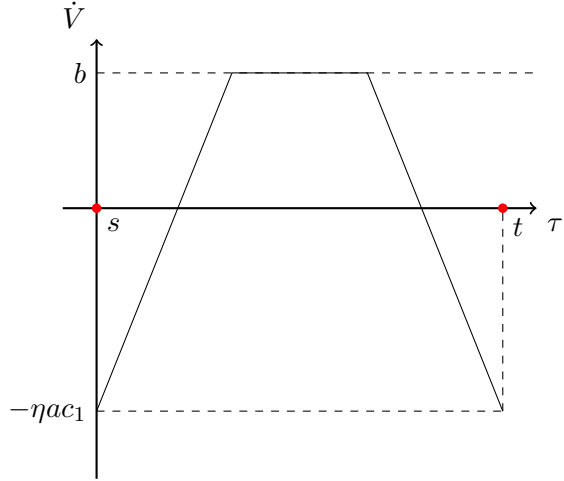

An illustration of the upper bound of $\dot V$ over $[s,t]$ is plotted in Figure \ref{fig:vdott}, corresponding to the trajectory in Figure \ref{fig:traj}. If  $t-s \leq 2\frac{b+\eta a c_1}{\alpha \bar L_0}$ the functions to be minimized in \eqref{eqn_trapezoid} have only one switching point at $\frac{t-s}{2}$, and
\begin{align*}
\Delta V_{(s,t)}&\leq \int_0^{\frac{t-s}{2}} (\alpha \bar L_0 \tau -\eta a c_1)d\tau+\int_{\frac{t-s}{2}}^{t-s}(\alpha \bar L_0(t-s-\tau)-\eta a c_1) d\tau \\
&=2\int_0^{\frac{t-s}{2}}(\alpha \bar L_0\tau-\eta a c_1) d\tau=\alpha\bar L_0(\frac{t-s}{2})^2-2\eta a c_1(\frac{t-s}{2})\\
&=\frac{1}{4}\alpha\bar L_0(t-s)^2-\eta a c_1(t-s).
\end{align*}
If $(t-s) > 2\frac{b+\eta a c_1}{\alpha \bar L_0}$, there are two switching points: $\tau = \frac{b+\eta a c_1}{\alpha \bar L_0}$ and $\tau = t-s-\frac{b+\eta a c_1}{\alpha \bar L_0}$ so
we have
\begin{align*}
\Delta V_{(s,t)}\leq& \int_0^{\frac{b+\eta a c_1}{\alpha \bar L_0}} (\alpha \bar L_0 \tau -\eta a c_1)d\tau +\int_{\frac{b+\eta a c_1}{\alpha \bar L_0}}^{t-s-\frac{b+\eta a c_1}{\alpha \bar L_0}}bd\tau\\
&+\int_{t-s-\frac{b+\eta a c_1}{\alpha \bar L_0}}^{t-s}(\alpha \bar L_0(t-s-\tau)-\eta a c_1)d\tau\\
=&2\int_0^{\frac{b+\eta a c_1}{\alpha \bar L_0}}(\alpha \bar L_0 \tau-\eta a c_1)d\tau +\int_{\frac{b+\eta a c_1}{\alpha \bar L_0}}^{t-s-\frac{b+\eta a c_1}{\alpha \bar L_0}}b\, d\tau \\
=&\alpha\bar L_0(\frac{b+\eta ac_1}{\alpha \bar L_0})^2-2\eta a c_1(\frac{b+\eta ac_1}{\alpha \bar L_0})\\
&+b\left((t-s)-2(\frac{b+\eta ac_1}{\alpha \bar L_0})\right)\\
=&b(t-s)-(\frac{b+\eta ac_1}{\alpha \bar L_0})^2.
\end{align*}
The two bounds are collected to be the $\phi$ function as stated in the lemma.

\end{proof}

Now since we have assumed that $b< a c_1$ in the beginning, we can always pick an $\eta$ sufficiently close to $1$ to guarantee that
\begin{equation}\label{assumption3}
b<\eta ac_1.
\end{equation}
From now on we will assume that $\eta$  satisfies both \eqref{turning_radius} and \eqref{assumption3}. Notice that for the solution outside  $\Omega_\eta$, the almost Lyapunov function $V$ clearly is decreasing; therefore, Lemma \ref{lem:5} also leads us to the following conclusion:
\begin{Corollary}\label{cor:2}
Consider a solution $x(\tau)$ with $V(x(0))< c_2-h\epsilon^{\frac{1}{n}}-g\epsilon$. Let $(s,t)$ be a maximal connected component of $X_{\eta}$ such that $s\neq 0,t\neq T$. Assume also   $b< \eta ac_1$ and $\epsilon<\bar \epsilon$. Then $\Delta V_{(s,t)}\leq \phi(t-s)<0$.
\end{Corollary}
\begin{proof} We prove by induction under an additional assumption that there are finitely many connected components in any bounded subset of  $X_{\eta}$. The extension to the general
case  will be justified at the end of the proof.

Firstly, if $(t-s) \alpha \bar{L}_0<2(b+\eta ac_1)$, \eqref{assumption3} implies $(t-s)\alpha \bar{L}_0<4\eta ac_1$ and hence the first line in \eqref{phi} implies  $\phi(t-s)=\frac{1}{4} (t-s)^2\alpha \bar{L}_0 - (t-s)\eta ac_1<0$. Otherwise, \eqref{con:teta2} implies $\phi(t-s)=b (t-s)-\frac{(b+\eta ac_1)^2}{\alpha \bar{L}_0}<0$. Thus we always have $\phi(t-s)<0$.

 Let $(s,t)$ be the first connected component of $X_{\eta}$ on the left with $s>0$. If it is the first connected component on the left (i.e. there is no connected component starting at $\tau=0$) then $V(x(s))<V(x(0))< c_2-h\epsilon^{\frac{1}{n}}-g\epsilon$. If there is a connected component starting at $\tau=0$, say the interval $(0,t_0)$, then still
\begin{align*}
V(x(s)) \leq V(x(0)) + \Delta V_{(0,t_0)}   <  (c_2-g\epsilon-h\epsilon^{\frac{1}{n}})+\frac{g}{2}\epsilon   =c_2-\frac{g}{2}\epsilon+h\epsilon^{\frac{1}{n}}.
\end{align*}
Either way, $V(x(s)) < c_2-\frac{g}{2}\epsilon+h\epsilon^{\frac{1}{n}}$. Hence by Lemma \ref{lem:5}, the base case is true and we have $\Delta V_{(s,t)}\leq \phi(t-s)<0$. Assume towards induction that at some connected component denoted also $(s,t)$ we have $V(x(s))< c_2-\frac{g}{2}\epsilon-h$ and $\phi(t-s)<0$. Then
at the next connected component $(s^+,t^+)$ we have
\begin{align*}
V(x(s^+))=\left(V(x(s^+))-V(x(t))\right)+\Delta V_{(s,t)}+V(x(s))\\
\leq \phi(t-s)+V(x(s))  <  c_2-\frac{g}{2}\epsilon-h\epsilon^{\frac{1}{n}}
\end{align*}
and again by Lemma \ref{lem:5} we have $\Delta V_{(s^+,t^+)}\leq\phi(t^+-s^+) < 0$.

Now, we address the case when $X_{\eta}$ is arbitrary, not necessarily consisting of finitely many connected components. Consider any connected component $(s,t) \subset X_{\eta}$ excluding those which contain boundary points.
If the corresponding arc of the solution does not enter $\Omega_0$, then $V$ could only decrease and we declare this component for the purpose of this proof to be outside of $X_{\eta}$. Now consider any connected component of $X_{\eta}$ for which the corresponding solution enters $\Omega_0$. Then, by Lemma \ref{lem:2} such a connected component
must have a lower bound on its length. Thus, the number of connected components where $V$ might increase has to be finite on a bounded time interval and the above proof by induction applies.

\end{proof}

\subsection{Exponential bound when repeatedly passing through $\Omega_\eta$}\label{subsection}
Corollary \ref{cor:2} tells us that the Lyapunov function decreases each time the solution crosses $\Omega_\eta$. This does not yet guarantee convergence to a smaller set. We now want to find an exponential type bound on $V$. Define $k(t):\mathbb R_+\to\mathbb R$ by
\[
k(t):=\left\{\begin{array}{cc}
-\frac{1}{t}\ln\left(1+\frac{1}{c_2}\phi(t)\right)&\mbox{ if }\phi(t)>-c_2,\\
K&\mbox{ if }\phi(t)\leq -c_2.
\end{array}\right.
\]
where $\phi$ is defined in \eqref{phi} and $K$ is a sufficiently large positive constant. Note that $\phi(t)$ is continuous near 0 and  $\phi(0)=0$, so we can define $k(0)=\frac{\eta a c_1}{c_2}$ by extension via L'H\^opital's rule.
In addition, define
\[
\lambda(\epsilon):=\min_{0\leq \delta \leq \epsilon}k\left(\frac{g\delta}{b}\right).
\]
By this definition, $\lambda(\epsilon)$ is a non-increasing function on $[0,\bar\epsilon)$. On the one hand, we see from the proof of Corollary \ref{cor:2} that $\phi (t)<0$ for all $t\in(0,\frac{(b+\eta ac_1)^2}{\alpha \bar{L}_0b})$ and thus we have $k(t)>0$ for all $t\in[0,\frac{(b+\eta ac_1)^2}{\alpha \bar{L}_0b})$. In addition, because $\frac{g\bar\epsilon}{b}\leq\frac{(b+\eta ac_1)^2}{\alpha \bar{L}_0b}$ as in \eqref{con:teta2}, $\lambda(\epsilon)$ is also positive on $[0,\bar\epsilon)$. According to Corollary \ref{lem:last}, $t-s\leq\frac{g\epsilon}{b}$, which implies
\[
k(t-s)\geq \min_{0\leq \delta \leq \epsilon}k\left(\frac{g\delta}{b}\right)=\lambda(\epsilon).
\]

Next, we have
\begin{align}
\nonumber V(x(t))&=\Delta V_{(s,t)}+V(x(s))=V(x(s))\left(1+\frac{\Delta V_{(s,t)}}{V(x(s))}\right)\\
\nonumber &\leq V(x(s))\left(1+\frac{\phi(t-s)}{c_2}\right)= V(x(s))e^{-k(t-s)(t-s)}\\
&\leq V(x(s))e^{-\lambda (\epsilon)(t-s)}\label{must_have}
\end{align}
for any connected component of $(s,t)\subset X_{\eta}$ that does not contain the end  points $\tau=0$ or $\tau=t_{\rm max}$. From the second line to the third line the inequality  $\Delta V_{(s,t)}\leq \phi(t-s)<0$ was used.
We also have
\[
\lambda(\epsilon)\leq \lambda(0)=k(0)=\frac{\eta a c_1}{c_2}<\eta a
\]
for all $\epsilon\in [0,\bar\epsilon)$. Thus, when the solution is inside $\Omega_\eta$, it has a decay rate slower than when the solution is in $D\backslash\Omega_\eta$, which has decay rate faster than $\eta a$. We can modify $\lambda(\epsilon)$ so that it is a positive, continuous, strictly decreasing function on $[0,\bar\epsilon)$ with $\lambda(0)<\eta a$ and so the  inequality \eqref{must_have} still holds.

As a result, for any $s,s'\in(0,T)\backslash{\rm int} X_\eta$, we have
\[
V(x(s')) \leq V(x(s))  e^{-\lambda(\epsilon)(s'-s)}.
\]
This exponential decaying bound suggests that $T$ cannot be infinite, otherwise for $s'\in(0,T)\backslash{\rm int} X_\eta$ and large enough we will have $V(x(s'))<c_1+h\epsilon^{\frac{1}{n}}$, implying $x(s')\not\in D^*$, and such $s'$ always exists when $T$ is infinite because the possible connected component containing $T$ has maximal length of $\frac{g\epsilon}{b}$.

Take an arbitrary $t\in[0,T]$. Recall that by  Lemma \ref{lem:5} for any connected components of $X_{\eta}$, even those that contain the end points $0$ and $t$, we still have the bound $\Delta V \leq \frac{g}{2}\epsilon$. Therefore, taking into account boundary components, we have
\begin{equation}\label{exp_bound1}
V(x(t)) \leq (V(x(0))+ \frac{g}{2}\epsilon) e^{-\lambda(\epsilon)(s'-s)}+\frac{g}{2}\epsilon
\end{equation}
where $s'=t$  if $t\not\in X_\eta$ , or $s'$ is the left boundary point of the connected component of $X_{\eta}$ containing $t$ otherwise; $s=0$ if $s\not\in X_\eta$, or $s$ is the right boundary point of the connected component of $X_{\eta}$ containing $0$ otherwise. From \eqref{exp_bound1} we directly see that
\begin{equation}\label{property2}
V(x(t)\leq V(x(0))+g\epsilon\quad\forall t\in[0,T].
\end{equation}
The first statement in the main Theorem follows from \eqref{property2} up to time $T$. In addition, by Corollary \ref{lem:last},
\[
s \leq  \frac{g\epsilon}{b},\quad t-s'\leq \frac{g\epsilon}{b}
\Rightarrow s'-s \geq t-2 \frac{g\epsilon}{b}.
\]
Substituting these expressions  into \eqref{exp_bound1}, we have
\begin{equation}\label{exp_bound2}
V(x(t))\leq e^{2\lambda(\epsilon)\frac{g\epsilon}{b}}(V(x(0))+\frac{g}{2}\epsilon)e^{-\lambda(\epsilon)t}+\frac{g}{2}\epsilon.
\end{equation}
This is also true for $t=T$. By definition of $T$ in \eqref{def:T} we see that $x(T)\in\partial D^*$ and because of the exponential decaying bound in \eqref{exp_bound2} so we must have $V(x(T))=c_1+h\epsilon^{\frac{1}{n}}$. The argument cannot proceed for $t>T$ because as $x(t)$ is outside of $D^*$, Lemma ~\ref{cor:5} cannot be applied and $B_{\gamma_\eta}(x(t))$ may not be contained in $D$ even if $\dot V(x(t))\leq -\eta aV(x(t))$; consequently the estimation of the sweeping volume, based on the bounds $\underline L_0,L_1$ etc. defined over $D$ is no longer valid. Nevertheless, once the solution returns to the lower boundary of $D^*$ such that $V(x(t))=c_1+h\epsilon^{\frac{1}{n}},$ it can be again treated as a new solution starting from $x(0) \in D$ with $\ V(x(0))< c_2-h\epsilon^{\frac{1}{n}}-g\epsilon$ and by the same analysis above we know that it can have an overshoot of $g\epsilon$ at most. This proves the second statement in the main theorem.

%

\section{Global uniform asymptotic stability  result by almost Lyapunov function}\label{sec:GUAS}
Our Theorem \ref{thm1} gives a local convergence property so that any solution in the domain converges to a lower level set. It is  often desirable to establish a  global convergence property so the solutions converge to a stable equlibrium. One typical stability property for autonomous systems is \textit{Global Uniform Asymptotic Stability} (GUAS), which means that the system is \textit{globally stable} in the sense that for any $\varepsilon>0$, there exists $\delta>0$ such that if $|x(0)|\leq \delta$, $|x(t)|\leq \varepsilon$ for all $t\geq 0$ and \textit{uniformly attractive} in the sense that for any $\delta>0,\kappa>0$, there exists $T=T(\delta,\kappa)$ such that whenever $|x(0)|\leq \kappa$, $|x(t)|\leq \delta$ for all $t\geq T$. We now try to transfer our study to a global result. To do that, instead of a fixed region $D$ defined by two constants $c_1,c_2$, we let the band-shaped region be defined for any $c>0$:
\begin{equation}\label{band_D}
D(c):=\{x\in\mathbb R^n:c\leq V(x)\leq 2c\}.
\end{equation}
Following the definitions of $ b,\bar L_0,\underline L_0,L_1, M_1,M_2$ from \eqref{def:b},\eqref{bar_L0},\eqref{underline_L0},\eqref{L_1},\eqref{M_1},\eqref{M_2} over the region $D(c)$, we see that now all of them are functions of $c$. We present a global uniform asymptotic stability result derived using an almost Lyapunov function:

\begin{Theorem}\label{thm:GUAS}
Consider a system (\ref{def:sys}) with a globally Lipschitz right-hand side $f$, and a function $V:\mathbb R^n\to[0,+\infty)$ which is positive definite and $C^1$ with globally Lipschitz gradient. In addition assume $V(x)\geq k_0|x|^2$ for some $k_0>0$ and all $x\in\mathbb R^n$. For any $c>0$, let the region $D(c)$ be defined via (\ref{band_D}) and assume all of them are compact. Let $\Omega\subset \mathbb R^n$ be a measurable set such that $\dot V(x)<-aV(x)$ for all $x\in \mathbb R\backslash\Omega$ with some $a>0$. Assume $\sup_{c>0}\frac{b(c)}{ac} < 1$ where $b(c)$ is defined via \eqref{def:b} over $D(c)$. Let $\underline L_0(c)$ be defined via \eqref{underline_L0} over $D(c)$. Then there exist $K_1,K_2,K_3>0$ such that if $\textnormal{vol}(\Omega^*(c))<\min\{K_1\underline L_0(c)^n,K_2c^{\frac{n-1}{2}}\underline L_0(c),K_3c^{\frac{n}{2}}\}$ for all $c>0$ where $\Omega^*(c)$ is the largest connected component of $\Omega\cap D(c)$, the system \eqref{def:sys} is GUAS.
\end{Theorem}

Before giving the proof of Theorem \ref{thm:GUAS}, let us discuss the validity and some variations of the assumptions of this theorem first. If we know that the system is globally stable or the working space is some compact set in $\mathbb R^n$ instead of $\mathbb R^n$ itself, then we can replace global Lipschitzness in $f$ and $V_x$ by local Lipschitzness as it is sufficient for the existence of uniform $L_1,M_2$, which will be used in the proof.
The assumption $V(x)\geq k_0|x|^2$ is quite general since all quadratic Lypunov function satisfies this assumption. Other assumptions are merely same as or the general versions of the assumptions in Theorem \ref{thm1}. The non-vanishing assumption is also reflected in the theorem statement that if $f$ vanishes at any state which is different from the origin, $\underline L_0(c)=0$ for some $c>0$ and this theorem becomes inconclusive.
\begin{proof}
The idea of the proof is to repeatedly apply Theorem \ref{thm1} over the region $D(c)$ for any $c>0$ and show that $V(x(t))$ is bounded and will decrease by a factor of fixed factor each time.

First of all, globally Lipschitz $f$ and $V_x$ mean there exist $k_1,k_2>0$ such that
\begin{align*}
L_1\leq k_1,\\
M_2\leq k_2,
\end{align*}
where $L_1,M_2$ are the global Lipschitz constants of $f,V_x$, respectively. In addition, if $x^*$ is the maximizer of $|f(x)|$ in $D(c)$,
\[
\bar L_0(c)=\max_{x\in D(c)}|f(x)|=|f(x^*)|=|f(x^*)-f(0)|\leq L_1|x^*-0|\leq k_1|x^*|\leq k_1\sqrt{\frac{V(x^*)}{k_0}}\leq k_1\sqrt{\frac{2c}{k_0}}.
\]
By similar argument we also have $M_1\leq k_2\sqrt{\frac{2c}{k_0}}$. Thus, $\alpha=M_1L_1+M_2\bar L_0\leq 2k_1k_2\sqrt{\frac{2c}{k_0}}$. Using $\eta\in(0,1)$, \eqref{gamma} in Lemma \ref{lem:2} becomes
\[
\gamma_\eta=\frac{(1-\eta)ac}{\alpha+\eta a M_1}\geq \frac{(1-\eta)ac}{2k_1k_2\sqrt{\frac{2c}{k_0}}+ a k_2\sqrt{\frac{2c}{k_0}}}=\frac{(1-\eta)a\sqrt{k_0}}{\sqrt{2}(2k_1+a)k_2}c^{\frac{1}{2}}=(1-\eta)Kc^{\frac{1}{2}}=:\gamma^*,
\]
where $K:=\frac{a\sqrt{k_0}}{\sqrt{2}(2k_1+a)k_2}$ is a constant. For each $c>0$, pick $\eta(c)\in(\frac{1}{2},1)$ such that
\begin{equation}\label{1-eta}
1-\eta(c)<\min\left\{\frac{\underline L_0(c)}{2k_1K\sqrt{c}},1-\sup_{c>0}\frac{b(c)}{ac}\right\},
\end{equation}
This can be done as the arguments in the min function on the right side of \eqref{1-eta} are always positive (the positiveness of the second argument is given by the theorem assumption). This also means that,
\begin{equation}\label{gamma*}
\gamma^*<\min\left\{\left(1-\sup_{c>0}\frac{b(c)}{ac}\right)Kc^{\frac{1}{2}},\frac{\underline L_0}{2k_1}\right\},
\end{equation}

which tells us that by a proper choice of $\eta(c)$ satisfying \eqref{1-eta}, $\gamma^*$ will be the minimum of two increasing functions of $c,\underline L_0$, respectively. Also by definition we know $\gamma^*\leq\gamma_\eta$, so the result in Lemma \ref{cor:5} holds for $\gamma^*$ as well. In addition, the inequality between $\gamma^*$ and $\frac{\underline L_0}{2k_1}$ in \eqref{gamma*} tells that
\[
\gamma^*<\frac{\underline L_0}{2k_1}\leq \frac{\underline L_0}{2L_1}<\frac{\underline L_0}{L_1},
\]
and the inequality between $1-\eta(c)$ and $1-\sup_{c>0}\frac{b(c)}{ac}$ in \eqref{1-eta} tells that
\[
\eta(c)>\sup_{c>0}\frac{b(c)}{ac}\Rightarrow b(c)<\eta(c)ac\quad\forall c>0.
\]
Therefore the bound \eqref{gamma*} guarantees that both \eqref{turning_radius} and \eqref{assumption3} are satisfied; $\gamma^*$ is indeed a valid sweeping tube radius and hence all the subsequent results still follow if we replace every $\gamma_\eta$ by $\gamma^*$. Now define
\begin{align}
\epsilon_3:&=\frac{\underline L_0(c)\text{vol}(B^{n-1}_{\gamma^*})}{4b}c,\label{e3}\\
\epsilon_4:&=\text{vol}(B_{r(c)}^n),\quad r(c)={\sqrt{\frac{k_0c}{32k_2^2}}}.\label{e4}
\end{align}
Then $\epsilon<\epsilon_3$ with $g$ substituted by its definition \eqref{g} implies
\[
g\epsilon<\frac{b\epsilon_3}{\underline L_0 \text{vol}(B^{n-1}_{\gamma^*})}< \frac{1}{4}c.
\]
On the other hand, $\epsilon<\epsilon_4$ with $h$ substituted by its definition \eqref{h} implies
\[
h\epsilon^{\frac{1}{n}}=M_1\left(\frac{\epsilon}{\chi(n)}\right)^{\frac{1}{n}}< k_2\sqrt{\frac{2c}{k_0}}\left(\frac{\epsilon_4}{\chi(n)}\right)^{\frac{1}{n}}=k_2\sqrt{\frac{2c}{k_0}}r(c)<\frac{1}{4}c.
\]
So we have both $g\epsilon$ and $h\epsilon^{\frac{1}{n}}$ bounded from above by $\frac{1}{4}c$ when $\epsilon$ is small enough.

Now for any initial state $x(0)\in \mathbb R^n$, we let $c=\frac{2}{3} V(x(0))$. Then $x_0\in D(c)$ and we try to apply Theorem \ref{thm1} on it. Notice that $V(x_0)=\frac{3}{2}c<2c-h\epsilon^{\frac{1}{n}}-g\epsilon$, thus the initial state satisfies the hypothesis. Hence we conclude from Theorem \ref{thm1} that for $\epsilon$ small enough, $V(x(t))\leq V(x(0))+g\epsilon\leq \frac{7}{4}c$ for all $t\geq 0$ and $V(x(t))\leq c+h\epsilon^{\frac{1}{n}}< \frac{5}{4}c$ for some $t\leq T(c,\epsilon)$. The global stability part is given by the first conclusion by letting $\delta =\frac{7}{6}\varepsilon$. The second conclusion tells that
\[
\frac{V(x(t))}{V(x(0))}<\frac{\frac{5}{4}c}{\frac{3}{2}c}=\frac{5}{6}
\]
Thus over each iteration $|x(t)|$ is decreased at least by a factor of $\frac{5}{6}$, in time at most $T$. We then reset time $t$ to be the initial time and can repeat the same argument. Thus while given $\delta$ and $\kappa$, the total number of iterations is $\lceil\frac{\ln\kappa-\ln\delta }{\ln 6-\ln 5}\rceil$ for a solution that starts from $\bar B_ \kappa^n(0)$ and converges to  $\bar B_\delta^n(0)$. The total time needed is bounded by the summation of $T(c,\epsilon)$'s of each iteration and hence for given $\epsilon$, it only depends on $\kappa,\delta$.

It remains to find how small $\epsilon$ needs to be; that is, find an expression of $\bar\epsilon$, which is the common lower bound of $\epsilon_1,\epsilon_2,\epsilon_3,\epsilon_4$, in terms of $c,\underline L_0$. Recall from \eqref{overall_volume} and \eqref{e2} that we have

\[
\epsilon_1=\text{vol}(B_{\gamma^*}^{n-1})\frac{2\underline{L}_0}{L_1}\left(\pi-\sin^{-1}(\frac{L_1\gamma_\eta^*}{\underline{L}_0})\right)\geq  \text{vol}(B_{\gamma^*}^{n-1})\underline{L}_0\frac{2}{k_1}\left(\pi-\sin^{-1}(\frac{1}{2})\right),
\]

\[
\epsilon_2=\text{vol}(B_{\gamma_\eta}^{n-1})\frac{\underline{L}_0(b+\eta ac)^2}{\alpha \bar{L}_0b}\geq \text{vol}(B_{\gamma^*}^{n-1})\underline{L}_0\frac{4b\eta ac}{2k_1^2k_2\frac{2c}{k_0}b}> \text{vol}(B_{\gamma^*}^{n-1})\underline{L}_0\frac{a}{2k_1^2k_2},
\]
where on the second line the assumption $\eta>\frac{1}{2}$ is used. Meanwhile, from \eqref{e3} we have
\[
\epsilon_3=\frac{\underline L_0(c)\text{vol}(B^{n-1}_{\gamma^*})}{4b}c> \text{vol}(B_{\gamma^*}^{n-1})\underline{L}_0\frac{1}{4a}.
\]
It is observed from the above inequalities that a common lower bound of $\epsilon_1,\epsilon_2,\epsilon_3$ is of the form $K_0\text{vol}(B_{\gamma^*}^{n-1})\underline{L}_0$ with some constant $K_0>0$. Recall from \eqref{gamma*} that $\gamma^*$ is chosen to be the minimum between two linear increasing functions of $\underline L_0,c^{\frac{1}{2}}$, respectively. Thus $\text{vol}(B_{\gamma^*}^{n-1})$ is the minimum between two linear increasing functions of $\underline L_0^{n-1},c^{\frac{n-1}{2}}$, respectively. As a result,
\[
\min\{\epsilon_1,\epsilon_2,\epsilon_3\}\geq\min\{K_1\underline L_0^n,K_2c^{\frac{n-1}{2}}\underline L_0\}
\]
In addition, \eqref{e4} means that $\epsilon_4$ is a linear function of $c^{\frac{n}{2}}$. Put them together, we have
\[
\bar \epsilon:=\min\{K_1\underline L_0^n,K_2c^{\frac{n-1}{2}}\underline L_0,K_3c^{\frac{n}{2}}\}\leq \min\{\epsilon_1,\epsilon_2,\epsilon_3,\epsilon_4\}
\]
This $\bar \epsilon$ is the upper bound of $\epsilon$ in Theorem \ref{thm1}. As a result, as long as $\textnormal{vol}(\Omega^*(c))< \bar\epsilon$ for all $c>0$ where $\Omega^*(c)$ is the largest connected component of $\Omega\cap D(c)$, the system \eqref{def:sys} is GUAS.
\end{proof}

\section{Example and discussion}\label{nontrivialexample}

\subsection{Example}

The system (\ref{def:sys}) is explicitly defined as follows:
\begin{equation}\label{def:sys2}
\left(\begin{array}{c}\dot{x_1}\\\dot{x_2}\end{array}\right)=f(x)=
\left(\begin{array}{cc}-\lambda(x)&\mu\\-\mu&-\lambda(x)\end{array}\right)\left(\begin{array}{c}x_1\\x_2\end{array}\right)
\end{equation}
with
\[
\lambda(x)=1.01\min\left\{\frac{|x-x_c|}{\rho},1\right\}-0.01,\quad x_c=(0.8,0)^\top,\mu=2,\rho=0.01.
\]
\begin{figure}
\centering
  \includegraphics[scale=0.3]{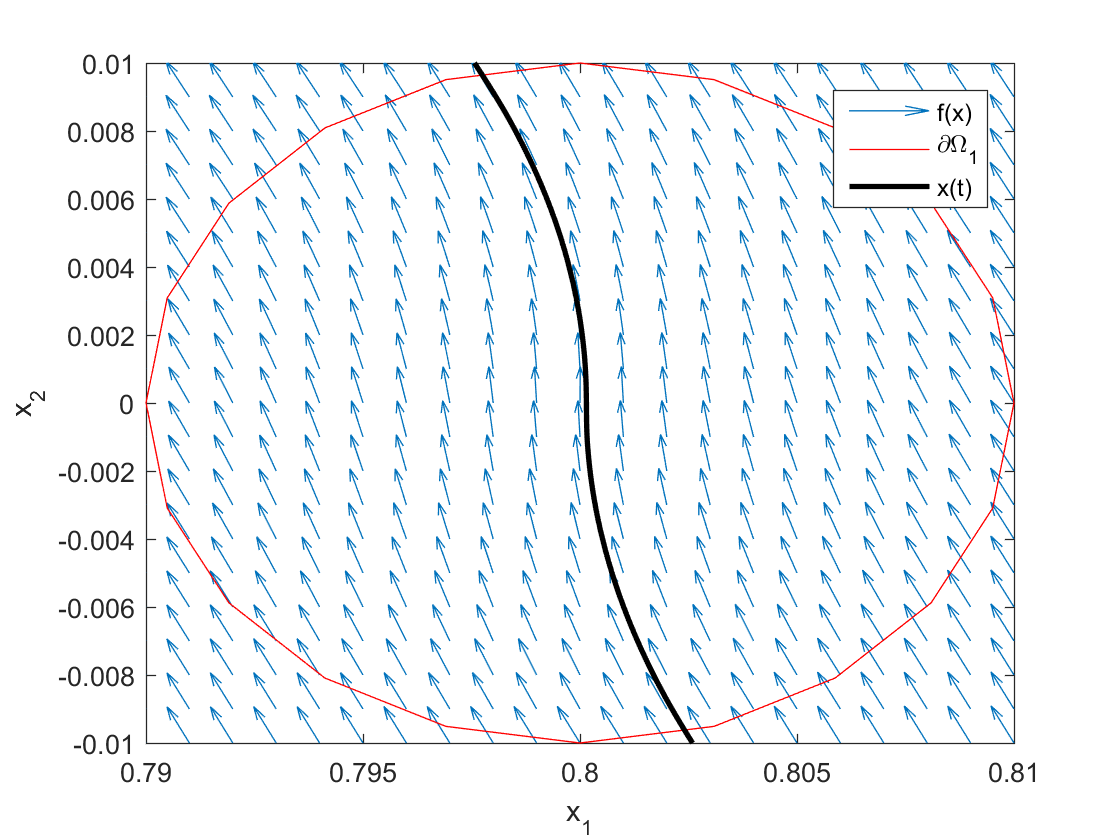}
  \caption{Local behavior of the example system}
  \label{example}
\end{figure}
The relevant part of the phase portrait for the vector field $f(x)$ with a solution $x(t)$ passing through is shown in Figure \ref{example}. Notice that the spiral-shaped vector field is distorted in the region of $B_\rho(x_c)$. The solution $x(t)$ passing through this region will temporarily move away from the origin when passing through $B_\rho(x_c)$. More explicitly, we consider the function
\begin{equation*}
V=|x|^2=x_1^2+x_2^2
\end{equation*}
as a candidate Lyapunov function.
Then
\begin{equation}\label{ex_Vdot}
\dot V(x)=2(x_1\dot x_1+x_2\dot x_2)=-2\lambda(x)(x_1^2+x_2^2).
\end{equation}
Notice that $\lambda(x)=1$ everywhere except in $B_\rho(x_c)$. Outside this ball  $B_\rho(x_c)$ the system is linear and satisfies the decay condition $\dot V=-2V$. When $x(t)$ is very close to $x_c$, $\lambda(x)$ becomes negative and $\dot{V}$ becomes positive. Hence for this system $\Omega_0\neq\emptyset$ and $V$ is not a Lyapunov function for this system but only an almost Lyapunov function. Nevertheless, we will  show by our theorem that convergence to $0$ takes place as the effect of $\Omega$ is not strong. To do so, choose $d_1=0.7,d_2=1,c_1=d_1^2,c_2=d_2^2$. We find that
\begin{align*}
|f(x)|&=\sqrt{f(x)^\top f(x)}\\
&=\sqrt{\begin{pmatrix}
x_1&x_2
\end{pmatrix}\begin{pmatrix}
-\lambda(x)&-\mu\\\mu&-\lambda
\end{pmatrix}\begin{pmatrix}
-\lambda(x)&\mu\\-\mu&-\lambda
\end{pmatrix}\begin{pmatrix}
x_1\\x_2
\end{pmatrix}}\\
&=\sqrt{(\lambda^2(x)+\mu^2)(x_1^2+x_2^2)}\\
&=|x|\sqrt{(\lambda^2(x)+\mu^2)}
\end{align*}
Hence on the set $D=\{x:d_1\leq|x|\leq d_2\}$,
\[
\bar{L}_0=d_2\times\sqrt{\max\lambda(x)^2+\mu^2}=\sqrt{5},
\]
\[
\underline{L}_0=d_1\times\sqrt{\min\lambda(x)^2+\mu^2}=1.4.
\]
The parameter $L_1$ was computed  numerically to be $90.78$. Since $V_x(x)=2(x_1,x_2)$,
\[
M_1=2d_2=2,
\]
\[
M_2=2.
\]
In addition, from \eqref{ex_Vdot} we see that
\[
b=-2\min_{x\in D}\lambda(x)|x|^2.
\]
The minimum is achieved at $x=x_c$ and it is computed to be
\[
b=0.0128.
\]
Naturally pick $a=2$ so that $\Omega=B_\rho(x_c)$. Thus,
\[
 \epsilon=\text{vol}({\Omega})=\pi\rho^2\approx 3.14\times10^{-4}.
 \]
 Also note that this $\Omega$ is completely inside $D$.\\
Pick $\eta=0.6$. It can be calculated that
\[
\alpha=M_1L_1+\bar L_0M_2\approx 186,
\]
\[
\gamma_\eta=\frac{(1-\eta)ac_1}{\alpha+\eta a M_1}\approx 0.0021\leq0.0154=\frac{\underline{L}_0}{L_1},
\]
so \eqref{turning_radius} is satisfied. In addition,
\[
\eta a c_1=0.588> b,
\]
so \eqref{assumption3} is also satisfied. Hence $\eta=0.6$ is large enough. We can then compute $\bar\epsilon$:
\begin{align*}
&\epsilon_1=4\gamma_\eta\frac{\underline{L}_0}{L_1}\left(\pi-\sin^{-1}(\frac{L_1\gamma_\eta}{\underline{L}_0})\right)\approx 3.86\times 10^{-4}.\\
&\epsilon_2=\frac{2\gamma_\eta \underline{L}_0(b+\eta ac_1)^2}{\alpha \bar{L}_0b}\approx 3.95\times 10^{-4},\\
&\Rightarrow \bar\epsilon=\max\{\epsilon_1,\epsilon_2\}=3.95\times 10^{-4}
\end{align*}
Indeed we have
\[
\epsilon<\bar\epsilon.
\]
So all the hypothesis in Theorem 1 hold. Meanwhile,
\begin{align*}
&h=M_1\gamma_\eta\approx0.0042\\
&g\epsilon=\frac{b\text{vol}(\Omega_1)}{2\gamma_\eta\underline{L}_0}\approx 6.9\times 10^{-4}\ll c_2-c_1
\end{align*}
The conclusions in Theorem 1 tell us that the system will converge to the set $\{x:V(x)\leq c_1+h+g\epsilon\}\approx B_{0.7044}(0)$ if it starts at $x_0$ with $V(x_0)\leq c_2-h-g\epsilon\approx0.9951$. \\

\subsection{Discussion of the Example}
Firstly, because our $V$ is chosen to be quadratic and we know from the earlier discussion in Section \ref{subsection} that the convergence of $V$ is exponential, we can further conclude that the convergence of the solution to the ball $B_{0.7044}(0)$ is exponentially fast. In addition, since $\dot V(x)=-2V(x)$ for all $x\in B_{0.7044}(0)\cup\{x:V(x)>0.9951\}$, the system is in fact globally exponentially stable.

It is important to note, as discussed earlier, that in this example $\Omega_0\neq\emptyset$. By continuity of $\dot V$ as a function of states, we know that there will be $x'\in\Omega$ such that $V_x(x')\cdot f(x')=\dot V(x')=0$ (which is in fact on $\partial \Omega_0$). If we don't require the vector field to be non-vanishing, then since $V_x(x)=2x\neq 0$ for all $x\in D$, we either have $f(x')=0$ or $V_x(x')$ is orthogonal to $f(x')$. In the first case $x'$ is an equilibrium of the system and we will have a solution $x(t)\equiv x'$, which would not converge to a smaller set and hence the conclusion in Theorem \ref{thm1} is no longer true. This indicates that the additional assumption of non-vanishing (which results in the positive bound $\underline{L}_0$) is indeed crucial to establishing the convergence result.

Recall that the significance of our main theorem appears when there are multiple ``bad regions'' with the volume of each of them  bounded above. For instance, by modifying the vector field of the above example such that $\Omega$ consists of multiple $B_\rho(x_i)$ regions distributed in D with $|x_i|=0.8$ for all $i$, our main theorem is still applicable and will lead to the same conclusion.

Nevertheless, the obtained $\bar\epsilon$ appears to be rather conservative. One can observe in the above example that the radius of the sweeping ball is quite small as $\gamma_\eta\approx\frac{1}{5}\rho$; as a result, $\bar\epsilon$ which is proportional to $\text{vol}(B_{\gamma_\eta}^{n-1})$ becomes very small. It is not hard to see from the proofs of Lemma \ref{lem:0}, \ref{lem:1} and \ref{lem:2} that $\gamma_\eta$ is a very coarse bound on the radius of the largest ball that is contained in $\Omega$. More careful analysis can be done on tightening $\gamma_\eta$; however, this may require additional information about system dynamics. Our current assumptions on the system, on the other hand, are rather general.

In addition, once $\eta$ is chosen, a sweeping ball of constant radius is employed for the analysis. We can make $\gamma_\eta$ time-varying based on the level set of $\Omega_\eta$ that $x$ is in. Since it is known that the radius of the sweeping ball becomes larger when $\dot{V}$ becomes positive, $\bar \epsilon$ will be larger and this modification should yield a better result. However, difficulties arise in converting the bound \eqref{overall_length} on the length of a particular trajectory to a \eqref{overall_volume}-like bound on the volume of $\Omega_1$.

\section{Conclusion}\label{conclusion}
We presented a result (Theorem \ref{thm1}) which establishes convergence of system trajectories from a given set to a smaller set, based on an almost Lyapunov function which is known to decrease along solutions on the complement of a set of small enough volume. The result is established by tracking the change of Lyapunov function value when the solution passes through this set of small volume and finding an upper bound on the volume swept out by a tubular neighborhood along the solution before it can achieve an overall gain in its Lyapunov function value. With some knowledge of the structure of the system dynamics, it is shown that convergence will still hold even if there is some temporary gain in Lyapunov function value. We have also developed Theorem \ref{thm:GUAS} that under mild assumptions of the system, the result of Theorem \ref{thm1} can be iterated so that when the volume where $\dot V$ is not negative enough is small, the system can still be shown to be GUAS.

\section{Acknowledgement}
 S.L. and D.L.  were supported by the NSF grant CMMI-1662708 and the AFOSR grant FA9550-17-1-0236. V.Z. was partially supported by  a grant from the Simons Foundation \# 278840.

\bigskip

\bibliography{reference}

\begin{thebibliography}{10}
\providecommand{\url}[1]{#1}
\csname url@samestyle\endcsname
\providecommand{\newblock}{\relax}
\providecommand{\bibinfo}[2]{#2}
\providecommand{\BIBentrySTDinterwordspacing}{\spaceskip=0pt\relax}
\providecommand{\BIBentryALTinterwordstretchfactor}{4}
\providecommand{\BIBentryALTinterwordspacing}{\spaceskip=\fontdimen2\font plus
\BIBentryALTinterwordstretchfactor\fontdimen3\font minus
  \fontdimen4\font\relax}
\providecommand{\BIBforeignlanguage}[2]{{%
\expandafter\ifx\csname l@#1\endcsname\relax
\typeout{** WARNING: IEEEtran.bst: No hyphenation pattern has been}%
\typeout{** loaded for the language `#1'. Using the pattern for}%
\typeout{** the default language instead.}%
\else
\language=\csname l@#1\endcsname
\fi
#2}}
\providecommand{\BIBdecl}{\relax}
\BIBdecl

\bibitem{KH02}
H.~Khalil, \emph{Nonlinear Systems, 3rd ed.}\hskip 1em plus 0.5em minus
  0.4em\relax Prentice Hall, 2002.

\bibitem{RD65}
R.~D. Driver, ``Methods of {A}. {M}. {L}yapunov and their application
  ({V}.{I}.{Z}ubov),'' \emph{SIAM Review}, vol.~7, no.~4, pp. 570--571, 1965.

\bibitem{FC01}
F.~Camilli, L.~Gr\"une, and F.~Wirth, ``A generalization of {Z}ubov's method to
  perturbed systems,'' \emph{SIAM Journal on Control and Optimization},
  vol.~40, no.~2, pp. 496--515, 2001.

\bibitem{SD02}
S.~Dubljević and N.~Kazantzis, ``A new {L}yapunov design approach for
  nonlinear systems based on {Z}ubov's method,'' \emph{Automatica}, vol.~38,
  no.~11, pp. 1999 -- 2007, 2002.

\bibitem{BR00}
B.~Reznick, ``Some concrete aspects of {H}ilbert's 17th problem,''
  \emph{Contemporary Mathematics}, pp. 251--272, 2000.

\bibitem{GC11}
G.~Chesi, \emph{Domain of Attraction: Analysis and Control via SOS
  Programming}.\hskip 1em plus 0.5em minus 0.4em\relax London:Springer, 2011.

\bibitem{GB12}
G.~Blekherman, P.~A. Parrilo, and R.~R. Thomas, \emph{Semidefinite optimization
  and convex algebraic geometry}.\hskip 1em plus 0.5em minus 0.4em\relax
  Society for Industrial and Applied Mathematics : Mathematical Optimization
  Society, 2012.

\bibitem{RT12}
R.~Tempo, G.~Calafiore, and F.~Dabbene, \emph{Randomized Algorithms for
  Analysis and Control of Uncertain Systems, 2nd ed.}\hskip 1em plus 0.5em
  minus 0.4em\relax London: Springer, 2012.

\bibitem{MV02}
M.~Vidyasagar, \emph{A Theory of Learning and Generalization: With Applications
  to Neural Networks and Control Systems}.\hskip 1em plus 0.5em minus
  0.4em\relax Secaucus, NJ, USA: Springer-Verlag New York, Inc., 1997.

\bibitem{DL14}
D.~Liberzon, C.~Ying, and V.~Zharnitsky, ``On almost {L}yapunov functions,'' in
  \emph{2014 IEEE 53th Conference on Decision and Control (CDC)}, Dec 2014, pp.
  3083--3088.

\bibitem{SL16}
S.~Liu, D.~Liberzon, and V.~Zharnitsky, ``On almost {L}yapunov functions for
  non-vanishing vector fields,'' in \emph{2016 IEEE 55th Conference on Decision
  and Control (CDC)}, Dec 2016, pp. 5557--5562.

\bibitem{AB69}
A.~Butz, ``Higher order derivatives of {L}iapunov functions,'' \emph{IEEE
  Transactions on Automatic Control}, vol.~14, no.~1, pp. 111--112, February
  1969.

\bibitem{AA11}
A.~A. Ahmadi and P.~A. Parrilo, ``On higher order derivatives of {L}yapunov
  functions,'' in \emph{Proceedings of the 2011 American Control Conference},
  June 2011, pp. 1313--1314.

\bibitem{CR89}
R.~Courant and F.~John, \emph{Introduction to Calculus and Analysis, Volume
  II}.\hskip 1em plus 0.5em minus 0.4em\relax Springer New York, 1989.

\bibitem{RF06}
R.~Foote, ``The volume swept out by a moving planar region,'' in
  \emph{Mathematics Magazine, Volume 79, Number 4}.\hskip 1em plus 0.5em minus
  0.4em\relax Mathematical Association of America, Oct 2006.

\bibitem{JS08}
J.~M. Sullivan, ``Curves of finite total curvature,'' in \emph{Discrete
  Differential Geometry}, A.~I. Bobenko, J.~M. Sullivan, P.~Schr{\"o}der, and
  G.~M. Ziegler, Eds.\hskip 1em plus 0.5em minus 0.4em\relax Basel:
  Birkh{\"a}user Basel, 2008.

\end{thebibliography}
\bibliographystyle{IEEEtran}

\appendix
\begin{center}
\begin{Large}
{\bf Appendix }
\end{Large}
\end{center}

\section{Previous result}\label{earlier_result}
We provide a slightly different result in this section. In this case the region of interest is defined as:
\begin{equation}\label{modified_D}
D:=\{x\in \mathbb R^n:V(x)\leq c\}
\end{equation}
Notice that in this case $D$ is defined with the origin included, in contrast to the the one defined for Theorem 1 which excludes a neighborhood of origin. Here is the theorem statement:
\begin{Theorem} \label{thm2}\cite{DL14}
Let $\rho:(0,+\infty)\to(0,+\infty)$ be the relation such that
\[
\text{vol}(B_{\rho(\epsilon)})=\epsilon
\]
Consider the system \eqref{def:sys} with a locally Lipschitz right-hand side $f$, and a function $V$ which is positive definite and $\mathcal C^1$ with locally Lipshitz gradient. Let the region $D$ be defined via \eqref{modified_D} and assume that it is compact. Assume that \eqref{lcvdot} holds. Then there exist a constant $\bar \epsilon>0$ and a continuous, strictly increasing function $\bar R$ on $[0,\bar \epsilon]$ with $\bar R(0)=0$ such that for every $\epsilon\in(0,\bar\epsilon)$, if $\text{vol}(\Omega)<\epsilon$, then for every initial condition $x_0\in D$ with
\[
V(x_0)<c-2M_1\rho(\epsilon)
\]
where $M_1$ is defined by \eqref{M_1}, the corresponding solution $x(\cdot)$ of \eqref{def:sys} with $x(0)=x_0$ has the following properties:
\begin{enumerate}
\item $V(x(t))\leq V(x_0)+2M_1\rho(\epsilon)$ for all $t\geq 0$ (and hence $x(t)\in D$ for all $t\geq 0$).
\item $V(x(T))\leq \bar R(\epsilon)$ for some $T\geq 0$.
\item $V(x(t))\leq \bar R(\epsilon)+2M_1\rho(\epsilon)$ for all $t\geq T$.
\end{enumerate}
\end{Theorem}
The proof of Theorem \ref{thm2} is established by a perturbation argument which compares a given system trajectory with nearby trajectories that lie entirely in $D\backslash\Omega$ and trades off convergence speed of these trajectories against the expansion rate of the distance to them from the given trajectory. For more details of the proof of Theorem \ref{thm2}, please refer to \cite{DL14}. Notice that in the special case when $\dot V (x) \leq-aV (x)$ for all $x \in D$ (which implies that $\epsilon$ can be any arbitrarily small positive number), Theorem \ref{thm2} reduces to Lyapunov's classical asymptotic stability theorem. On the other hand, we cannot recover asymptotic stability from Theorem \ref{thm1} when $\text{vol}(\Omega)=0$ simply because a neighborhood of origin is taken away from $D$. At first sight one may think the main Theorem \ref{thm1} in this paper has some drawbacks as it requires extra conditions (existence of positive $\underline L_0, b$) to hold than Theorem \ref{thm2}; meanwhile, the result of Theorem \ref{thm1} seems to be weaker than that of Theorem \ref{thm2} due to the existence of gap $h$ in all three statements, which unlike $g\epsilon$ in Theorem \ref{thm1} or $R(\epsilon)$ in Theorem \ref{thm2} and does not vanish as $\epsilon$ goes to $0$. Nevertheless, we need to point out that the two $\bar \epsilon$'s in both theorems are very different; in fact the $\bar\epsilon$ in Theorem \ref{thm2} is very conservative compared with that of Theorem \ref{thm1}. In order to fulfill the condition in Theorem \ref{thm2}, we need $\text{vol}(\Omega)<\bar \epsilon$. However, we failed to construct a non-trivial example with $\dot V(x)>0$ for some $x\in D$ while maintaning that inequality. This is left as an open question in \cite{DL14}. An interesting observation is that by perturbing the system dynamics without increasing the Lipschitz constant, which is used in computing $\bar \epsilon$, an unstable equilibrium can be constructed away from the origin. There will be contradiction if Theorem \ref{thm2} is applicable to such a system because a solution starting at that unstable equilibrium will not move, contrary to what is concluded from the theorem that the solution will be attracted to a neighborhood of the origin. On the other hand, if we try to apply Theorem \ref{thm2} to the example in Section \ref{example}, through the procedure in \cite{DL14} we find that $\bar \epsilon <\pi \rho ^2$, thus Theorem \ref{thm2} is inconclusive. Hence we prefer to apply Theorem \ref{thm1} with a modified region $D$.

\section{Proof of Proposition \ref{non-overlap}: }
If a  space curve $x^*(s),s\in[0,\mathcal L]$ is closed ($x^*(0)=x^*(\mathcal L)$) and piecewise $C^2$, we set $I:=\{s\in[0,\mathcal L):\frac{d}{ds}x^*(s) \mbox{ does not exist}\}$. For each $s\in I$, we define the turning angle $\varphi_t(s)$ to be the oriented angle from the vector $\frac{d}{ds}x^*(s^-)$ (or $\frac{d}{ds}x^*(\mathcal L^-)$ if $s=0$) to the vector $\frac{d}{ds}x^*(s^+)$. Then total curvature is defined as
\[
K=\int_{s\in [0,\mathcal L)\backslash I}\kappa(s)ds+\sum_{s\in I}\varphi_t(s)
\]
In order to prove Proposition \ref{non-overlap}, two geometrical results are needed:
\begin{Lemma}[Fenchel's Theorem]
For any closed space curve $x(s)$,
\[
K\geq 2\pi
\]
and equality holds if and only if $x(s)$ is a convex planar curve.
\end{Lemma}
\begin{Lemma}[Schur's Comparison Theorem]
Suppose $C(s)$ is a plane curve with curvature $\kappa(s)$ which makes a convex curve when closed by the chord connecting its endpoints, and $C^*(s)$ is an arbitrary space curve of the same length with curvature $\kappa^*(s)$. Let $d$ be the distance between the endpoints of $C$ and $d^*$ be the distance between the endpoints of $C^*$. If $\kappa^*(s) \leq \kappa(s)$ then $d^* \geq d$.
\end{Lemma}

%

Suppose self-overlapping occurs between $N_{\rho_0}(x(t))$ and $N_{\rho_0}(x(s))$ for some $t>s$. We prove the proposition by showing that contradictions arise if $\mathcal L_s^t<2 \rho\left(\pi-\sin^{-1}(\frac{\rho_0}{\rho})\right)$.

Rewrite $\mathcal L_s^t=2\rho\theta$ for some $\theta\in\left(0,\pi-\sin^{-1}(\frac{\rho_0}{\rho})\right)$. Let $z\in N_{\rho_0}(x(t))\cap N_{\rho_0}(x(s))$. Denote the angle between vector $\ovec{z x(t)}$ and vector $\ovec{z x(s)}$  by $\phi_z$. Notice that the curve $x(\tau)$ over $[s,t]$ and the two vectors $\ovec{z x(t)},\ovec{z x(s)}$ form a closed curve. Evaluating the total curvature alone this closed curve and applying Fenchel's Theorem and realizing that the turning angles at $x(t)$, $x(s)$ are both $\frac{\pi}{2}$ because they are on the normal disks, and the fact that the turning angle at $z$ is the complement of $\phi_z$, we have

\begin{align*}
2\pi\leq K&=\left(\int_{x(t)}^{x(s)}\kappa(x)dx\right)+ \varphi_t(x(s))+\varphi_t(x(t))+\varphi_t(z)\\
&\leq \int_{x(t)}^{x(s)}\frac{1}{\rho}dx+ \varphi_t(x(s))+\varphi_t(x(t))+\varphi_t(z)\\
&= \frac{\mathcal L_s^t}{\rho}+\frac{\pi}{2}+\frac{\pi}{2}+(\pi-\phi_z).\\
\end{align*}
Therefore
\begin{equation}\label{phi&theta}
\phi_z\leq 2\theta.
\end{equation}

\noindent
Now we establish the contradiction in 3 different cases, based on the value of $\theta$: \\

\noindent
{\bf Case 1.} $\theta<\frac{\pi}{4}$. Notice that because $f(x(t)),f(x(s))$ are normal vectors of $N_{\rho_0}(x(t))$, $N_{\rho_0}(x(s))$, the angle between them is the same as the dihedral angle
between the two hyperplanes that contain the two normal disks, which is the maximal value of $\phi_z$ over all possible $z$ along the intersection of the two hyperplanes. Because \eqref{phi&theta} always holds for such $\phi_z$, it also holds for the maximum, hence in this case the angle between $f(x(t))$ and $f(x(s))$ is acute. Now because $\rho_0<\rho$, the velocity of each point on the normal disk $N(x(\cdot))$ is in the same direction as $ f(x(\cdot))$ when $N(x(\cdot))$ ``sweeps" with respect to time. Thus renaming $t$ by $\tau$ and using the earlier result of acute angle between $f(x(\tau))$ and $f(x(s))$, we see that the velocity of each point on the normal disk $N(x(\tau))$ has positive component in the $f(x(s))$ direction for all $\tau\in[s,t]$. In other words, the disk $N(x(t))$ moves away from $N(x(s))$ so self-overlapping is impossible. \\

\noindent
{\bf Case 2.} $\theta\in[\frac{\pi}{4},\frac{\pi}{2})$. In this case, compare the solution $x(\cdot)$ to a circular arc with constant curvature $\frac{1}{\rho}$ and same arc length of $2\rho\theta$. Notice that such a circular arc has central angle $2\theta$ and therefore the chord length is $2\rho\sin\theta$. By Schur's Comparsion Theorem,
\[
|x(t)-x(s)|\geq 2\rho\sin\theta\geq \sqrt{2}\rho.
\]
In addition, $z\in N_{\rho_0}(x(t))\cap N_{\rho_0}(x(s))$ means $|z-x(t)|\leq \rho_0<\rho,|z-x(s)|\leq \rho_0<\rho$. Thus $|z-x(t)|^2+|z-x(s)|^2<2\rho^2\leq |x(t)-x(s)|^2$, which not only means that $\phi_z$ is obtuse, but also implies that
\begin{align*}
\cos \phi_z=&\frac{|z-x(t)|^2+|z-x(s)|^2- |x(t)-x(s)|^2}{2|z-x(t)||z-x(s)|}\\
< \,\,  & \frac{\rho^2+\rho^2-(2\rho\sin\theta)^2}{2\rho^2}\\
=&\cos 2\theta.
\end{align*}
Hence $\phi_z>2\theta$, contradicting \eqref{phi&theta} so self-overlapping is impossible in this case. \\

\noindent
{\bf Case 3.} $\theta\in[\frac{\pi}{2},\pi-\sin^{-1}(\frac{\rho_0}{\rho}))$. In this case we repeat the same procedure of comparing the solution $x(\cdot)$ to a circular arc. Again Schur's Comparison Theorem tells us that
\[
|x(t)-x(s)|\geq 2\rho\sin\theta >2\rho\sin(\pi-\sin^{-1}(\frac{\rho_0}{\rho}))=2\rho_0.
\]
Because $x(t)$ and $x(s)$ are separated by more than $2\rho_0$, self-overlapping is impossible.

\end{document}